\documentclass[12pt,a4paper]{article}
\usepackage{graphicx}
\usepackage[]{amsmath,amssymb,amsthm}
\usepackage{paralist}

\newtheorem{theorem}{Theorem}[section]
\newtheorem{proposition}[theorem]{Proposition}
\newtheorem{lemma}[theorem]{Lemma}

\theoremstyle{definition}
\newtheorem{definition}[theorem]{Definition}
\theoremstyle{remark}

\newcommand{\R}{\mathbb{R}}
\newcommand{\N}{\mathbb{N}}

\newcommand{\Z}{\mathbb{Z}}

\newcommand{\inn}{\textnormal{in}}
\newcommand{\out}{\textnormal{out}}
\newcommand{\ind}{\textnormal{index}}

\newcommand{\E}{{\mathcal E}}
\newcommand{\B}{{\cal B}}
\newcommand{\dom}{\textnormal{dom}}
\newcommand{\F}{\mathcal{F}}
\newcommand{\G}{\mathcal{G}}

\begin{document}

\begin{center}
{\large{\bf 
Stability in simple heteroclinic networks in $\R^4$}}\\
\mbox{} \\
\begin{tabular}{cc}
{\bf Sofia B.\ S.\ D.\ Castro$^{\dagger, *}$} & {\bf Alexander Lohse$^{\ddagger}$} \\
{\small sdcastro@fep.up.pt} & {\small alexander.lohse@math.uni-hamburg.de}
\end{tabular}

\end{center}

\noindent $^{*}$ Corresponding author. Phone: + 351 220 426 373. Fax: +351 225 505 050.

\noindent $^{\dagger}$ Faculdade de Economia and Centro de Matem\'atica, Universidade do Porto, Rua Dr.\ Roberto Frias, 4200-464 Porto, Portugal.

\noindent $^{\ddagger}$ Fachbereich Mathematik, Universit\"at Hamburg, Bundesstra{\ss}e 55, 20146 Hamburg, Germany.

\vspace{1cm}

\begin{abstract}
We describe all heteroclinic networks in $\R^4$ made of simple heteroclinic cycles of types $B$ or $C$, with at least one common connecting trajectory. For networks made of cycles of type $B$, we study the stability of the cycles that make up the network as well as the stability of the network. We show that even when none of the cycles has strong stability properties the network as a whole may be quite stable. We prove, and provide illustrative examples of, the fact that the stability of the network does not depend {\em a priori} uniquely on the stability of the individual cycles.
\end{abstract} 


\vspace{.3cm}

\vspace{2cm}

\section{Introduction}

Ever since symmetry provided a way of constructing robust heteroclinic cycles, the study of stability properties of such objects has been of interest. See, for instance, the work of Krupa and Melbourne \cite{KrupaMelbourne95a, KrupaMelbourne95b,KrupaMelbourne2004}, Melbourne \cite{Melbourne1991} and, more recently, that of Podvigina and Ashwin \cite{PodviginaAshwin2011} and Podvigina \cite{Podvigina2012}. 

It is clear that when joining heteroclinic cycles to produce a network none of the cycles can be asymptotically stable. In this instance, several intermediate notions of stability have been introduced by Melbourne \cite{Melbourne1991}, Brannath \cite{Brannath}, Podvigina and Ashwin \cite{PodviginaAshwin2011}. The strongest of these is {\em predominant asymtotic stability}, p.a.s.\ in the sequel. A p.a.s.\ invariant object will attract a set of points in its neighbourhood big enough so that the invariant object may be visible in numerical simulations and experiments. We therefore focus on this kind of stability.

Recently, Podvigina and Ashwin \cite{PodviginaAshwin2011} have defined the stability index of a trajectory and the second author \cite{Lohse} has been making use of this to determine necessary and sufficient conditions for p.a.s.\ of heteroclinic cycles. We make extensive use of these results to study not only stability of each cycle in a network, but, by combining stability indices at a local and global level, to determine which cycle attracts more trajectories in a neighbourhood of the network, that is, which cycle may be visible in simulations. A similar question has been addressed by
Kirk and Silber \cite{KS} who have shown by example that when looking at attraction properties of cycles which are part of a heteroclinic network, these are patent even when none of the cycles is stable. We contribute to a better understanding of this phenomenon by providing an exhaustive treatment of the stability and attractiveness of simple cycles making up a network in $\R^4$. It is important to notice that we need the stability index at two levels in order to draw conclusions.

Additionally, we show that a network may be p.a.s.\ while only one or both of the two cycles it is made of are non-p.a.s.. 
While this might induce the thought that joining heteroclinic cycles into heteroclinic networks has some stabilizing effect, we also present a non-p.a.s.\ network constructed from two cycles, only one of which is p.a.s..
Thus, the stability of each cycle does not condition the stability of the heteroclinic network meaning that the study of the stability of heteroclinic networks and cycles in heteroclinic networks is still at its onset.

\medskip

In the next section we provide definitions and results which will be used in subsequent sections. Some of the results were obtained by the second author alone and their proof will appear elsewhere. See also Lohse \cite{Lohse}. Section \ref{simple_networks} is devoted to the description of all possible networks in $\R^4$ consisting of simple cycles of types $B$ and $C$ with at least one common connecting trajectory. There are only three such networks, one of them being the one extensively studied by Kirk and Silber \cite{KS}. In Appendix A we construct and study the remaining network involving only type $B$ cycles. Section \ref{stability_cycles} provides results for calculating the stability indices along connections of a cycle. In section \ref{B3B3_network} we study a network made of two cycles of type $B^-_3$, already looked at in \cite{KS}. We determine the stability of each cycle and  decide when the network as a whole has some kind of stability. In such cases, we establish which cycle is more likely to be observed in the network. On the one hand, this provides an alternative, more systematic, description of the results obtained in \cite{KS} (subsection \ref{another-view}). On the other hand, we address some unstudied cases, thus completing the study of competition between cycles in the network of \cite{KS} (subsections \ref{positive-transverse} and \ref{stabilizing}). Section \ref{stability_network} concludes by building on results of the previous section to show that the stability of the network does not depend {\em a priori} solely on the stability of each cycle. 

\section{Preliminary results} \label{preliminaries}

Consider a vector field in $\R^4$ described by a set of differential equations $\dot{y}=f(y)$, where $f$ is $\Gamma$-equivariant for some finite subgroup of $O(4)$, $\Gamma$, that is,
$$
f(\gamma .y)=\gamma. f(y), \; \; \; \forall \; \gamma \in \Gamma \; \; \forall \; y \in \R^4.
$$
A {\em heteroclinic cycle} consists of equilibria $\xi_i$, $i=1, \hdots, m$ together with trajectories which connect them:
$$
[\xi_i \rightarrow \xi_{i+1}] \subset W^u(\xi_i) \cap W^s(\xi_{i+1}) \neq \emptyset.
$$
We assume $\xi_{m+1}=\xi_1$ and use $X$ to represent the heteroclinic cycle. It is well-known that demanding that the connection $[\xi_i \rightarrow \xi_{i+1}]$ be of saddle-sink type in an invariant subspace is enough to ensure robustness of the cycle.

Heteroclinic cycles are called {\em simple} if the connections between consecutive equilibria are contained in a $2$-dimensional subspace. We use the definition of \cite[p.\ 1181]{KrupaMelbourne2004}:
let $\Sigma_j \subset \Gamma$ be an isotropy subgroup and let $P_j=\textnormal{Fix}(\Sigma_j)$. Assume that for all $j=1,\hdots , m$ the connection $[\xi_j \rightarrow \xi_{j+1}]$ is a saddle-sink connection in $P_j$. Write $L_j=P_{j-1}\cap P_j$. A robust heteroclinic cycle $X \subset \R^4\backslash \{ 0\}$ is {\em simple} if 
\begin{enumerate}
	\item[(i)] dim$P_j=2$ for each $j$;
	\item[(ii)] X intersects each connected component of $L_j\backslash \{ 0\}$ in at most one point.
\end{enumerate}

We focus on simple cycles with no double eigenvalues, which seems to have been silently assumed in most of the literature. Recently, these cycles have been renamed {\em very simple} by Podvigina and Chossat \cite{PC2013}. We shall retain the designation {\em simple} and say that a heteroclinic network is {\em simple} if it is made of simple heteroclinic cycles.

 A classification of simple cycles into types A, B and C appears in Chossat {\em et al.} \cite{CKMS} who are concerned with bifurcations of cycles. The same classification is used in the context of stability of cycles by Krupa and Melbourne \cite{KrupaMelbourne2004} and Podvigina and Ashwin \cite{PodviginaAshwin2011}. Some type B and C cycles are grouped into a further type, Z, which appears in Podvigina \cite{Podvigina2012}. We use the original classification into types A, B and C, which we reproduce here from  \cite{KrupaMelbourne2004}.

\begin{definition}[Definition 3.2 in Krupa and Melbourne \cite{KrupaMelbourne2004}] 
Let $X \subset \R^4$ be a simple robust heteroclinic cycle.
\begin{enumerate}
	\item[(i)]  $X$ is of {\em type A} if $\Sigma_j=\Z_2$ for all $j$.
	
	\item[(ii)]  $X$ is of {\em type B} if there is a fixed-point subspace $Q$ with dim$Q=3$, such that $X \subset Q$.
	
	\item[(iii)]  $X$ is of {\em type C} if it is neither of type A nor of type B.
\end{enumerate}
\end{definition}

We are at present going to consider networks consisting of cycles only of types $B$ and/or $C$. Networks involving cycles of type $A$ will be treated separately.

For the cycles of types $B$ and $C$, further notation is used to indicate the order of the cycle, as well as whether $-I$ (minus the identity) is an element of $\Gamma$. For instance, $B_3^-$ represents a cycle of type $B$ with three nodes such that $-I \in \Gamma$, whereas $B_2^+$ represents a cycle of type $B$ with two nodes such that $-I \notin \Gamma$.

Criteria for the asymptotic stability of the cycles, depending on the eigenvalues of the vector field at each equilibrium, have been established by Krupa and Melbourne \cite{KrupaMelbourne95a,KrupaMelbourne2004}. When more than one cycle is put together in a heteroclinic network, none of the cycles is asymptotically stable. Instead, intermediate notions of stability have been introduced by Melbourne \cite{Melbourne1991}, Brannath \cite{Brannath}, Kirk and Silber \cite{KS}. We point out that the notion of ``essential asymptotic stability'' is used differently in Melbourne \cite{Melbourne1991} and in Brannath \cite{Brannath}. The definition of Melbourne \cite{Melbourne1991} is taken up by Podvigina and Ashwin \cite{PodviginaAshwin2011}, whereas that of Brannath \cite{Brannath} is renamed ``predominant asymptotic stability'' by Podvigina and Ashwin \cite{PodviginaAshwin2011}.
We use Podvigina and Ashwin \cite{PodviginaAshwin2011} as a reference for those concepts relevant for our results.

In the following, we denote by $B_{\varepsilon}(X)$ an $\varepsilon$-neighbourhood of a (compact, invariant) set $X \subset \R^n$. We write $\B(X)$ for the basin of attraction of $X$, i.e. the set of points $x \in \R^n$ with $\omega(x) \subset X$. For $\delta>0$ the $\delta$-local basin of attraction is $\B_\delta(X):=\{x \in \B(X) \mid \phi_t(x) \in B_\delta(X) \: \forall t>0  \}$, where $\phi_t(.)$ is the flow generated by the system of equations. By $\ell(.)$ we denote Lebesgue measure.

The following is the strongest intermediate notion of stability.

\begin{definition}[Definition 4 in Podvigina and Ashwin  \cite{PodviginaAshwin2011}]
A compact invariant set $X$ is called {\em predominantly asymptotically stable (p.a.s.)} if it is asymptotically stable relative to a set $N \subset \R^n$ with the property that
\begin{align*}
\lim\limits_{\varepsilon \to 0} \frac{\ell(B_{\varepsilon}(X) \cap N)}{\ell(B_\varepsilon(X))} = 1.
\end{align*}
\end{definition}

The same authors have introduced the following stability index as a means of quantifying the attractiveness of a compact, invariant set $X$. See Definition 5 and section 2.3 in Podvigina and Ashwin \cite{PodviginaAshwin2011}.

\begin{definition}
For $x \in X$ and $\varepsilon, \delta >0$ define 
\begin{align*}
\Sigma_\varepsilon(x)&:=\frac{\ell(B_\varepsilon(x) \cap \B(X))}{\ell(B_\varepsilon(x))}, \qquad \Sigma_{\varepsilon,\delta}(x):=\frac{\ell(B_\varepsilon(x) \cap \B_\delta(X))}{\ell(B_\varepsilon(x))}.
\intertext{Then the {\em stability index} at $x$ with respect to $X$ is defined to be}
\sigma(x)&:=\sigma_+(x)-\sigma_-(x),
\intertext{where}
\sigma_-(x)&:= \lim\limits_{\varepsilon \to 0} \left[ \frac{\textnormal{ln}(\Sigma_\varepsilon(x) )}{\textnormal{ln}(\varepsilon)}  \right], \qquad \sigma_+(x):= \lim\limits_{\varepsilon \to 0} \left[ \frac{\textnormal{ln}(1-\Sigma_\varepsilon(x) )}{\textnormal{ln}(\varepsilon)}  \right].
\intertext{The convention that $\sigma_-(x)=\infty$ if $\Sigma_\varepsilon(x)=0$ for some $\varepsilon>0$ and $\sigma_+(x)=\infty$ if $\Sigma_\varepsilon(x)=1$ is introduced. Therefore, $\sigma(x) \in [-\infty, \infty]$. In the same way the {\em local stability index} at $x \in X$ is defined to be}
\sigma_{\textnormal{loc}}(x)&:=\sigma_{\textnormal{loc},+}(x)-\sigma_{\textnormal{loc},-}(x),
\intertext{with}
\sigma_{\textnormal{loc},-}(x):= \lim\limits_{\delta \to 0} &\lim\limits_{\varepsilon \to 0} \left[ \frac{\textnormal{ln}(\Sigma_{\varepsilon,\delta}(x))}{\textnormal{ln}(\varepsilon)}  \right], \: \sigma_{\textnormal{loc},+}(x):= \lim\limits_{\delta \to 0} \lim\limits_{\varepsilon \to 0} \left[ \frac{\textnormal{ln}(1-\Sigma_{\varepsilon,\delta}(x))}{\textnormal{ln}(\varepsilon)}  \right].
\end{align*}
\end{definition}

 \begin{figure}[!htb]
 \centerline{
 \includegraphics[width=0.7\textwidth]{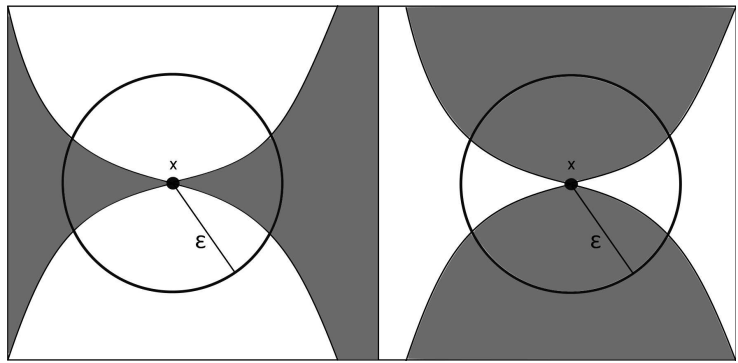}}
 \caption{Geometry of the basin of attraction (shaded region) at $x$ when $\sigma(x)<0$ (left), $\sigma(x)>0$ (right)}\label{stabindex}
 \end{figure}

The stability index $\sigma(x)$ quantifies the local extent (at $x \in X$) of the basin of attraction of $X$. If $\sigma(x)>0$, then in a small neighbourhood of $x$ an increasingly large portion of points is attracted to $X$, see Figure \ref{stabindex} (right), where the grey area is $\B(X)$. If on the other hand $\sigma(x)<0$, then the portion of such points goes to zero as the neighbourhood shrinks, also shown in Figure \ref{stabindex} (left).

Podvigina and Ashwin prove (see Theorem 2.2 in \cite{PodviginaAshwin2011}) that both $\sigma(x)$ and $\sigma_{\textnormal{loc}}(x)$ are constant along trajectories. This allows us to characterize the attraction properties of a heteroclinic cycle in terms of the stability index by calculating only a finite number of indices. Moreover, Podvigina and Ashwin also show that the calculation of the indices can be simplified by restricting to a transverse section (see Theorem 2.4 in \cite{PodviginaAshwin2011}).

Given Theorem \ref{local_index} below we calculate only local stability indices, which is why from section \ref{simple_networks} onwards we drop the subscript {\em loc}. These may be calculated using a cycle or the whole network as the invariant set $X$. In order to distinguish these two cases, we write $\sigma$ or $\sigma^c$ when the stability index is calculated with respect to a cycle and refer to it as a $c$-index. If the stability index is calculated with respect to the whole network we write $\sigma^n$ and refer to it as a $n$-index.

We now state two elementary results, that help us determine stability indices with respect to heteroclinic networks in section \ref{stability_cycles}.

\begin{lemma}\label{index_cycle-network}
Let $X$ be a heteroclinic network and $C \subset X$ a heteroclinic cycle. Then for all $x \in C$ we have $\sigma^n(x) \geq \sigma^c(x)$.
\end{lemma}
\begin{proof}
The proof is straightforward and can be found in \cite{Lohse}.
\end{proof}

\begin{lemma}\label{index_calculation}
Let $X \subset \R^n$ be a heteroclinic cycle (or network) and $x \in X$ a point on a connecting trajectory. Suppose that for all points $y=(y_1,...,y_n) \in B_\varepsilon(x)$, stability with respect to $X$ depends only on their $(y_1,y_2)$-components. Furthermore, assume that
\begin{align*}
&\B(X)\cap B_\varepsilon(x)=B_\varepsilon(x) \setminus \bigcup_{m \in \N} \E_m,
\intertext{where $\E_m$ are non-empty, disjoint sets of the form}
&\E_m =\left\{ y \in B_\varepsilon(x) \; \bigg\vert \; k_my_1^{\alpha_m} \leq y_2 \leq \hat{k}_my_1^{\alpha_m} \right\},
\intertext{with constants $k_m, \hat{k}_m >0$. Suppose that $(\alpha_m)_{m \in \N}$ is bounded away from $1$ and not all $\alpha_m$ are negative. Then with $\alpha_{max}:=\textnormal{max} \{ \alpha_m \mid 0<\alpha_m<1 \}$ and $\alpha_{min}:=\textnormal{min} \{ \alpha_m \mid \alpha_m>1 \}$ we have}
&\sigma(x)=-1+\textnormal{min} \left\{\frac{1}{\alpha_{max}}, \alpha_{min}   \right\} >0.
\end{align*}
\end{lemma}

\begin{proof}
For each $\alpha_m$ there is $\varepsilon>0$ small enough such that by straightforward integration we obtain
\begin{align*}
&\ell(\E_m \cap B_\varepsilon(x)) =
\begin{cases}
c_m \varepsilon^{n-1+\alpha_m } \qquad &\text{if} \quad \alpha_m > 1\\
\tilde{c}_m \varepsilon^{n-1+\frac{1}{\alpha_m}} \qquad &\text{if} \quad \alpha_m < 1.
\end{cases}
\intertext{Here (and in the following) we use $c_m,\tilde{c}_m>0$ to group together all constant terms. Since $\ell(B_\varepsilon(x))$ is of order $\varepsilon^n$ this yields}
&\Sigma_\varepsilon(x) =1- \sum\limits_{\alpha_m < 1} \tilde{c}_m \varepsilon^{-1+\frac{1}{\alpha_m}} - \sum\limits_{\alpha_m > 1} c_m \varepsilon^{-1+\alpha_m}.
\end{align*}

Thus, $\sigma_-(x)=0$, and calculating the limit $\sigma_+(x)$ gives the claimed result. For more details see \cite{Lohse}.
\end{proof}

When looking for the stability of the network the following result is very useful.

\begin{theorem}[see \cite{Lohse}] \label{local_index}
Let $X \subset \R^n$ be a heteroclinic cycle or network with finitely many equilibria and connecting trajectories. Suppose that the local stability index $\sigma_{\textnormal{loc}}(x)$ exists and is not equal to zero for all $x \in X$. Then $X$ is predominantly asymptotically stable if and only if $\sigma_{\textnormal{loc}}(x)>0$ along all connecting trajectories.
\end{theorem}

We note that the result does not hold if we consider global stability indices instead of local ones. For global stability indices we have the following:
\begin{lemma}
Let $[\xi_i \rightarrow \xi_j]$ be a common connecting trajectory between two non-homoclinic cycles constituting a simple heteroclinic network $X$ in $\R^4$. Suppose that for at least one of the cycles the return maps are contractions. Let $\sigma_{ij}$ and $\tilde{\sigma}_{ij}$ be the global stability indices for each cycle and $\sigma^n_{ij}$ the global stability index with respect to the whole network. We have $\sigma^n_{ij}>0$ if and only if $\sigma_{ij}>0$ or $\tilde{\sigma}_{ij}>0$.
\end{lemma}

\begin{proof}
That $\sigma_{ij}>0 \Rightarrow \sigma^n_{ij}>0$ and $\tilde{\sigma}_{ij}>0 \Rightarrow \sigma^n_{ij}>0$ is Lemma \ref{index_cycle-network}.

Assume $\sigma^n_{ij}>0$. For a point $x$, that contributes to the index $\sigma^n_{ij}$, we have $\omega(x) \subset X$, so $\omega(x)$ is compact, non-empty and connected, leaving three possibilities:
\begin{compactitem}
\item[(a)] $\omega(x)$ is an equilibrium.
\item[(b)] $\omega(x)$ is one of the cycles.
\item[(c)] $\omega(x)$ is the whole network $X$.
\end{compactitem}    
The set of points for which (a) holds is the union of the stable manifolds of the equilibria and thus of measure zero. Case (c) does not occur: the trajectory through $x$ would have to follow around both cycles infinitely many times, which is impossible since for at least one of the cycles the return maps are contractions. So almost all $x$ with $\omega(x) \subset X$  fall into case (b). Therefore, one of the cycles has a large enough basin of attraction to make the index with respect to only this cycle positive.
\end{proof}

It will be shown in case (ii) of Proposition \ref{KS_lemma3}  and Lemma \ref{KS_lemma3-network} that the local stability index $\sigma^n_{ij}$ may be positive even though both local $c$-indices $\sigma_{ij}$ or $\tilde{\sigma}_{ij}$ are negative.

\section{Simple networks in $\R^4$} \label{simple_networks}

This section is concerned with the construction of all possible simple networks, involving cycles of types $B$ and/or $C$, with at least one common connecting trajectory in $\R^4$. The demand of one common connecting trajectory excludes networks made of cycles with only equilibria in common. These will be studied elsewhere.

Because we want to concentrate on the dynamics associated with the network, we consider only dynamical systems for which there are no critical elements other than the origin and the equilibria in the network.

Given the definition of cycles of types A, B and C it is easily seen that, in a $4$-dimensional space, a cycle of type A and a cycle of type B or C cannot exist simultaneously in the same network. This is due to the fact that when a cycle is of type A, no element of the symmetry group acts as a reflection on $\R^4$ (see Corollary 3.5 in \cite{KrupaMelbourne2004}), whereas for B and C cycles there are always reflections. 

Type B cycles can be put together to form a heteroclinic network in two ways: the network studied by Kirk and Silber \cite{KS} made of two $B_3^-$-cycles and the one we construct in appendix A made of two $B_2^+$-cycles. There is only one way to combine cycles of different types, $B$ and $C$, to form a network, as the following proposition shows. 

\begin{proposition}
Let $X$ be a simple heteroclinic network in $\R^4$ that consists of two non-homoclinic cycles of type $B$ or $C$, which have at least one connecting trajectory in common. Suppose that there are no critical elements other than the origin and the equilibria belonging to the cycles.
Then the only possible networks are of type $(B_2^+,B_2^+)$, $(B_3^-,B_3^-)$ and $(B_3^-,C_4^-)$.
\end{proposition}

\begin{proof}
According to \cite[section 3.2]{KrupaMelbourne2004}, there are four distinct non-homoclinic type B and C cycles and they can exist only under equivariance of the system $\dot y=f(y)$ with respect to the following symmetry groups:
\begin{align*}
B_2^+ (\Z_2^3), \quad B_3^- (\Z_2^4), \quad C_2^- (\Z_2 \ltimes \Z_2^4), \quad C_4^- (\Z_2^4)
\end{align*}
Therefore, the only cycles of different type that may exist simultaneously are $B_3^-$ and $C_4^-$, under equivariance with respect to $\Z_2^4$.

The proof consists of two steps. In step $1$, we exclude the existence of some combinations of the above cycles. In step $2$, we show how we can combine cycles of types $B_3^-$ and $C_4^-$ to produce a network.

\paragraph{Step $1$:} We show that the combinations $(C_2^-,C_2^-)$ and $(C_4^-,C_4^-)$ are not possible, starting with $(C_4^-,C_4^-)$. Suppose we have a system with a $C_4^-$-cycle joining equilibria $\xi_1 \to \xi_2 \to \xi_3 \to \xi_4 \to \xi_1$, which, without loss of generality, lie on the respective coordinate axes. Now, also without loss of generality, there are three possibilities to introduce a second $C_4^-$-cycle with at least one common connecting trajectory:
\vspace{.05cm}
\begin{compactitem}
\item [(a)] add a connection $\xi_4 \to \xi_1$;
\item [(b)] add an equilibrium $\xi_*$ and connections $\xi_3 \to \xi_* \to \xi_1$;
\item [(c)] add two equilibria $\xi_*, \xi_{**}$ and connections $\xi_2 \to \xi_* \to \xi_{**} \to \xi_1$.
\end{compactitem}
\vspace{.3cm}
In case (a), the new connection $\xi_4 \to \xi_1$ has to lie in the coordinate plane $P_{14}$, so the phase portrait in this plane looks like that in Figure \ref{case-a}. Applying the Poincar\'e-Bendixson Theorem within the invariant plane $P_{14}$, one of the following holds:
\vspace{.05cm}
\begin{compactitem}
\item [(i)] $\xi_4$ is connected to $\xi_1$ by a two-dimensional set of trajectories.
\item [(ii)] There exists another equilibrium or periodic orbit inside $P_{14}$.
\end{compactitem}
\vspace{.3cm}
\begin{figure}[!htb]
\centerline{
\includegraphics[width=0.4\textwidth]{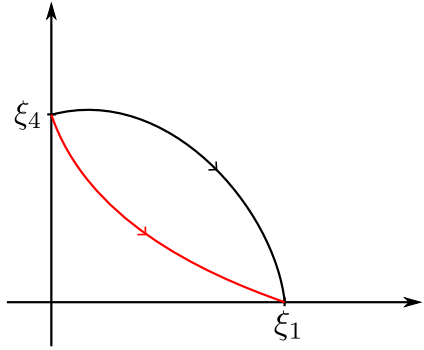}}
\caption{Phase portrait in $P_{14}$ for case (a).}
\label{case-a}
\end{figure}

Case (i) does not occur since, for a simple robust cycle the connection is of saddle-sink type and therefore $1$-dimensional in $P_{14}$. Case (ii) is excluded by our assumption. Hence, a connection added as in case (a) is not possible.

In case (b), the new equilibrium $\xi_*$ must lie in a one-dimensional fixed-point subspace. The only such subspaces are the coordinate axes. So $\xi_*$ must lie on the $y_4$-axis, since a $C_4^-$-cycle is not contained in a three-dimensional subspace. Thus, the phase portrait in $P_{14}$ looks like that in Figure \ref{case-b}. By a Poincar\'e-Bendixson argument similar to the one above, case (b) is impossible as well.

\begin{figure}[!htb]
\centerline{
\includegraphics[width=0.4\textwidth]{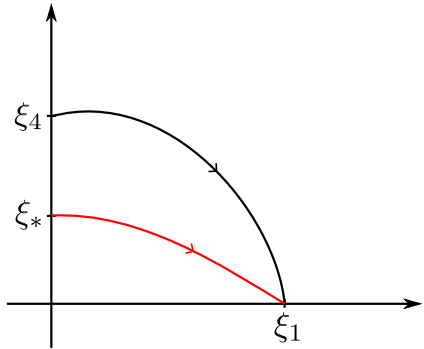}}
\caption{Phase portrait in $P_{14}$ for case (b).}
\label{case-b}
\end{figure}

In case (c), there are two subcases: (c-i) $\xi_*$ lies on the $y_3$-axis and $\xi_{**}$ on the $y_4$-axis. (c-ii) $\xi_*$ lies on the $y_4$-axis and $\xi_{**}$ on the $y_3$-axis. For (c-i), the phase portrait in $P_{14}$ looks exactly like the one in case (b), replacing $\xi_*$ with $\xi_{**}$. For (c-ii), dynamics in $P_{34}$ are shown in Figure \ref{case-c}. Again, a Poincar\'e-Bendixson argument yields that case (c) is not possible.

\begin{figure}[!htb]
\centerline{
\includegraphics[width=0.4\textwidth]{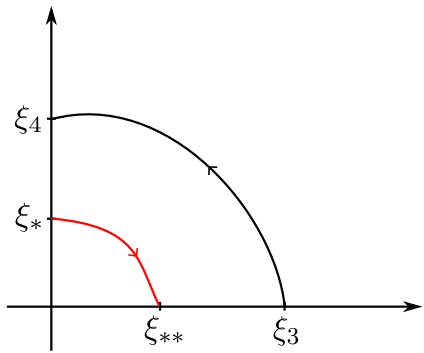}}
\caption{Phase portrait $P_{34}$ for case (c-ii).}
\label{case-c}
\end{figure}

The reasoning for a $(C_2^-,C_2^-)$-network is similar. A single cycle of type $C_2^-$ occupies the whole of $\R^4$ in the same way that a $C_4^-$-cycle does. There are also four equilibria, only now there are two pairs which are related by symmetry. Thus, analogous to the above, it follows that no additional $C_2^-$-cycle can be introduced to the system through adding connections and/or equilibria.

\paragraph{Step $2$:}  A $(B_3^-,C_4^-)$-network may be put together in the following way. Suppose we have a system $\dot y=f(y)$, equivariant under the action of $\Z_2^4$, with a heteroclinic cycle of type $C_4^-$. As above we assume it consists of four equilibria $\xi_i$ on the $x_i$-axis, joined by connecting trajectories in the coordinate planes in the following way: $\xi_1 \to \xi_2 \to \xi_3 \to \xi_4 \to \xi_1$. It is impossible to introduce a $B_3^-$-cycle to this system by adding an equilibrium (and two connections), for the same reasons as before. However, the existence of the $C_4^-$-cycle places no a priori restrictions on the dynamics in $P_{13}$. So we may consider the case, where there is a connection $\xi_3 \to \xi_1$ within $P_{13}$, making $\xi_1$ a sink in the three-dimensional coordinate space $S_{134}$ and $\xi_3$ expanding in the $x_1$- and $x_4$-directions. Then we have a second cycle $\xi_1 \to \xi_2 \to \xi_3 \to \xi_1$, contained in the three-dimensional fixed-point subspace $S_{123}$ and thus of type $B_3^-$. It has two connections in common with the $C_4^-$-cycle.
\end{proof}

We concentrate on the study of the $(B_3^-,B_3^-)$-network extending the work of \cite{KS}. The study of the $(B_2^+,B_2^+)$-network, because it follows closely that of the $(B_3^-,B_3^-)$-network, appears in appendix A.

\section{Stability}\label{stability_cycles}

In this section we make extensive use of the results that Podvigina and Ashwin \cite{PodviginaAshwin2011} obtained in subsection 4.2.1.\ of their paper concerning how to calculate the local stability indices for cycles of types $B_2^+$ and $B_3^-$. We transcribe their results, for ease of reference, in the following two lemmas. As noted in Section \ref{preliminaries}, we drop the subscript {\em loc}.
We recall, also from \cite{PodviginaAshwin2011}, that near $\xi_j$ the linearization of the vector field is given by
\begin{align}\label{linearization}
\begin{cases}
\dot{u}  &=  -r_j u\\
\dot{v}  &=  -c_j v\\
\dot{w}  &=  e_j w\\
\dot{z}  &=  t_j z,
\end{cases}
\end{align}
where $r_j$, $c_j$ and $e_j$ are positive but $t_j$ can have either sign. These eigenvalues are called radial, contracting, expanding and transverse, respectively.
The eigenvalues $-c_j$, $e_j$ and $t_j$ give rise to the quantities $a_j = c_j/e_j$ and $b_j = -t_j/e_j$ which will be used next. The function $f^\ind$ is used to express some of the indices, it can be found in \cite{PodviginaAshwin2011}, p. 905.

\begin{lemma}[Stability indices for type $B_2^+$ (\cite{PodviginaAshwin2011}, p.\  906)] \label{P&A_B2}
For a cycle of type $B_2^+$, the stability indices along connecting trajectories are as follows:
\begin{enumerate}
	\item[(i)]  If $b_1< 0$ and $b_2 < 0$, then the cycle is not an attractor and all stability indices are $-\infty$.
	\item[(ii)]  Suppose $b_1> 0$ and $b_2 > 0$.
	\begin{enumerate}
		\item[(a)]  If $a_1a_2 < 1$, then the cycle is not an attractor and all indices are $-\infty$.
		\item[(b)]  If $a_1a_2 > 1$, then the cycle is locally attracting and the stability indices are $+\infty$.
	\end{enumerate}
	\item[(iii)]  Suppose $b_1< 0$ and $b_2 > 0$.
	\begin{enumerate}
		\item[(a)]  If $a_1a_2 < 1$ or $b_1a_2+b_2 < 0$, then the cycle is not an attractor and all indices are $-\infty$.
		\item[(b)]  If $a_1a_2 > 1$ and $b_1a_2+b_2 > 0$,  then the stability indices are $\sigma_1 = f^{\textnormal{index}}(b_1,1)$ and $\sigma_2 = +\infty$.
	\end{enumerate}
\end{enumerate}
\end{lemma}

\begin{lemma}[Stability indices for type $B_3^-$ (\cite{PodviginaAshwin2011}, pp.\  906--907)] \label{P&A_B3}
For a cycle of type $B_3^-$, the stability indices along connecting trajectories are as follows:
\begin{enumerate}
	\item[(i)]  If $b_1< 0$,  $b_2 < 0$ and $b_3 < 0$, then the cycle is not an attractor and all stability indices are $-\infty$.
	\item[(ii)]  Suppose $b_1> 0$, $b_2 > 0$ and $b_3 > 0$.
	\begin{enumerate}
		\item[(a)]  If $a_1a_2a_3 < 1$, then the cycle is not an attractor and all indices are $-\infty$.
		\item[(b)]  If $a_1a_2a_3 > 1$, then the cycle is locally attracting and the stability indices are $+\infty$.
	\end{enumerate}
	\item[(iii)]  Suppose $b_1< 0$, $b_2 > 0$ and $b_3 > 0$.
	\begin{enumerate}
		\item[(a)]  If $a_1a_2a_3 < 1$ or $b_1a_2a_3 + b_3a_2 + b_2 < 0$, then the cycle is not an attractor and all indices are $-\infty$.
		\item[(b)]  If $a_1a_2a_3 > 1$ and  $b_1a_2a_3 + b_3a_2 + b_2 > 0$, then the stability indices are $\sigma_1 = f^{\textnormal{index}}(b_1,1)$,  $\sigma_2 = +\infty$ and $\sigma_3 = f^{\textnormal{index}}(b_3 + b_1a_3,1)$.
	\end{enumerate}
	\item[(iv)]  Suppose $b_1< 0$, $b_2< 0$ and $b_3 > 0$.
	\begin{enumerate}
		\item[(a)]  If $a_1a_2a_3 < 1$ or $b_2a_1a_3 + b_1a_3 + b_3 < 0$ or  $b_1a_2a_3 + b_3a_2 + b_2 < 0$, then the cycle is not an attractor and all indices are $-\infty$.
		\item[(b)]  If $a_1a_2a_3 > 1$ and $b_2a_1a_3 + b_1a_3 + b_3 > 0$ and  $b_1a_2a_3 + b_3a_2 + b_2 > 0$, then the stability indices are $\sigma_1 = \min\{f^{\textnormal{index}}(b_1,1),f^{\textnormal{index}}(b_1+b_2a_1,1)\}$,  $\sigma_2 = f^{\textnormal{index}}(b_2,1)$ and $\sigma_3 = +\infty$.
	\end{enumerate}
\end{enumerate}	
\end{lemma}
Note that compared to the statement in \cite{PodviginaAshwin2011}, in Lemma \ref{P&A_B3} (iv) (b) we have replaced $\sigma_3 = f^{\textnormal{index}}(b_3 + b_1a_3,1)$ by $\sigma_3=+\infty$. This is true since $b_2a_1a_3 + b_1a_3 + b_3 > 0$ implies $b_1a_3 + b_3 > -b_2a_1a_3 > 0$ and $f^{\ind}(\alpha,\beta)=+\infty$ for $\alpha,\beta>0$.

When the above results are applied to networks with a common connection, further simplifications arise as detailed next. Without loss of generality, we assume that the common connection is $[\xi_1 \rightarrow \xi_2]$ and that the stability indices for this connection depend on the values of $a_1$ and $b_1$. Name the cycles in the network $C_3$ and $C_4$.

\begin{lemma}
Let $X$ be a heteroclinic network made of two cycles of type $B_3^-$ and/or $C_4^-$, with one trajectory, $[\xi_1 \rightarrow \xi_2]$, common to both cycles. In the calculation of the stability indices, cases (i) and (ii) in Lemma \ref{P&A_B3} do not occur.
\end{lemma}

\begin{proof}
We start by showing that at least one of the quantities $b_i$ is negative for both cycles thus excluding (ii). Without loss of generality, let $b_1$ be the symmetric of the quotient between the transverse and expanding eigenvalues at the common node $\xi_2$. At this node, there exist one radial eigenvalue, one contracting eigenvalue (negative), an expanding eigenvalue (positive) and a transverse eigenvalue, which must be positive for both cycles. In fact, the expanding eigenvalue for cycle $C_3$ is the transverse eigenvalue for cycle $C_4$ and vice-versa. Therefore $b_1<0$.

It remains to show that (i) does not take place. Let, again without loss of generality, $b_3$ be the symmetric of the quotient between the transverse and expanding eigenvalues at the common node $\xi_1$. In this case, the transverse eigenvalue for cycle $C_3$ is the contracting eigenvalue for cycle $C_4$ and vice-versa. Therefore, $b_3>0$.
\end{proof}

Whether (iii) or (iv) occur in Lemma \ref{P&A_B3} depends on whether the transverse eigenvalue at the non-common node is negative or positive, respectively.

\begin{lemma}
Let $X$ be a heteroclinic network made of two cycles of type $B_2^+$ with one trajectory, $[\xi_a \rightarrow \xi_b]$, common to both cycles. In the calculation of the stability indices, cases (i) and (ii) in Lemma \ref{P&A_B2} do not occur. Furthermore, for one of the cycles the calculations fall in case (iii)(a) of Lemma \ref{P&A_B2}.
\end{lemma}

\begin{proof}
In this case, there are only two nodes, both of which are common to both cycles. Similar arguments to the above exclude (i) and (ii).

We denote by $e_{a2}$, $-c_{a3}$, $-c_{a4}$ the non-radial eigenvalues at $\xi_a$ and by $e_{b3}$, $e_{b4}$, $-c_{b2}$ those at $\xi_b$, where all quantities are positive. Note that in case (iii) of Lemma \ref{P&A_B2} $\xi_a$ and $\xi_b$ now take the role of $\xi_2$ and $\xi_1$, respectively. Then we obtain for cycle $C_3$
$$
b_{1}a_{2}+b_{2} = \frac{c_{a4}}{e_{a2}}-\frac{e_{b4}c_{a3}}{e_{a2}e_{b3}},
$$
and for cycle $C_4$
$$
b_{1}a_{2}+b_{2} = \frac{c_{a3}}{e_{a2}}-\frac{e_{b3}c_{a4}}{e_{a2}e_{b4}}.
$$
It is clear that the above quantities have opposite signs, so (iii)(a) applies.
\end{proof}

The above results provide the necessary information for determining the stability indices of each cycle {\em per se}, that is, when the basin of attraction $\B(X)$ takes $X$ to be only the cycle which we refer to by $c$-indices. We are also interested in calculating the stability indices when the basin of attraction $\B(X)$ takes $X$ to be the whole network, that is, the $n$-indices. In order to determine $\sigma_{ij}^n$, we use Lemmas \ref{index_cycle-network} and \ref{index_calculation}. 

\medskip

When joining together two or more cycles in a heteroclinic network, stability may be gained by one or both of the cycles. The $n$-index of connections of each cycle provides information about the relative stability of the cycles in a network. Cycles with higher $n$-indices are more stable and hence, more likely to be observed in experiments or simulations.
We are particularly interested in illustrating how the association of two cycles into a network can affect their stability properties. When one cycle is {\em per se} p.a.s., its indices will remain positive with respect to the whole network. This case does not provide an insight into the different stability when a cycle is considered on its own and when it is seen as part of a network. Kirk and Silber \cite{KS} have illustrated this point in their Lemma $3$. There, one of the cycles in the network has all stability indices equal to $-\infty$, while the other is either p.a.s.\ (case (i)) or has all stability indices but one equal to $+\infty$ (case (ii)). See also Propositions \ref{KS_lemma3} and \ref{KS_lemma3-network} below. In case (ii), neither cycle is p.a.s.\ but the network is p.a.s.. This illustrates the existence of what we call a {\em stabilizing mechanism} or {\em effect} in joining two cycles in a network. 

Note that infinite values for the stability index denote extreme stability characteristics and they are a feature in the networks we consider. In fact, we have the following generic results which are direct observations of Lemmas \ref{P&A_B2} and \ref{P&A_B3}, since all three possible networks involve a cycle of type $B_2^+$ or $B_3^-$. Note that the values of $f^{\textnormal{index}}(\alpha,1)$ in the lemmas are finite if and only if $\alpha <0$.

\begin{lemma}\label{generic_plus_infinity}
For the three possible heteroclinic networks in $\R^4$, at least one connecting trajectory has stability index equal to $+\infty$, unless all indices are equal to $-\infty$.
\end{lemma}

\begin{lemma}\label{generic_minus_infinity}
For the three possible heteroclinic networks in $\R^4$, if one trajectory has stability index equal to $-\infty$ then all connecting trajectories of that cycle have stability index equal to $-\infty$.
\end{lemma}

\section{The $(B_3^-,B_3^-)$-network}\label{B3B3_network}

This network has been studied, for a subset of parameter values, by Kirk and Silber \cite{KS}.
We briefly recall their notation. The two cycles have a common connection, namely, $[\xi_1 \longrightarrow \xi_2]$ and  are referred to as the $\xi_3$- and $\xi_4$-cycle, depending on which is the remaining node in the cycle. Locally near each node, $\xi_i$ ($i\neq 2$), the eigenvalues of the linearized vector field (\ref{linearization})
are now denoted by $-c_{ij}$, $-c_{ik}$, $-r_i$ and $e_{il}$ (contracting, transverse, radial and expanding, respectively). 
The second index provides information about the cycle we are considering. The choice presented indicates we are looking at node $\xi_i$ from the point of view of the cycle 
$$
[\xi_i \rightarrow \xi_l \rightarrow \xi_j \rightarrow \xi_i].
$$
Node $\xi_k$ does not belong to this cycle and hence $-c_{ik}$ denotes the transverse eigenvalues.
Near $\xi_2$, the eigenvalues are $-c_{21}$ (contracting), $-r_2$ (radial), $e_{23}$ (expanding with respect to the $\xi_3$-cycle and transverse with respect to the $\xi_4$-cycle) and $e_{24}$ (expanding with respect to the $\xi_4$-cycle and transverse with respect to the $\xi_3$-cycle). While in \cite{KS} all constants are positive, we allow some of them to become negative.

In what follows, maps from and to cross-sections of the flow along connections are extensively used. We list these maps in appendix B.

Associated to each cycle there are parameters depending on the eigenvalues, where those for the $\xi_3$-cycle are distinguished by a tilde, as follows:
$$
\begin{array}{lcl}
\rho=\dfrac{c_{42}c_{14}c_{21}}{e_{24}e_{41}e_{12}} & \mbox{\hspace{2cm}} & \tilde{\rho} = \dfrac{c_{32}c_{13}c_{21}}{e_{23}e_{31}e_{12}} \\
& & \\
\nu=\dfrac{e_{23}}{e_{24}}+\dfrac{c_{21}c_{43}}{e_{24}e_{41}}+\dfrac{c_{13}c_{42}c_{21}}{e_{41}e_{24}e_{12}}& \mbox{\hspace{2cm}} & \tilde{\nu}=\dfrac{e_{24}}{e_{23}}+\dfrac{c_{21}c_{34}}{e_{23}e_{31}}+\dfrac{c_{14}c_{32}c_{21}}{e_{31}e_{23}e_{12}} \\
& & \\
\delta = \dfrac{c_{43}}{e_{41}}+\dfrac{c_{13}c_{42}}{e_{12}e_{41}}-\dfrac{e_{23}c_{14}c_{42}}{e_{12}e_{41}e_{24}} & \mbox{\hspace{2cm}} & \tilde{\delta} = \dfrac{c_{34}}{e_{31}}+\dfrac{c_{14}c_{32}}{e_{12}e_{31}}-\dfrac{e_{24}c_{13}c_{32}}{e_{12}e_{31}e_{23}} \\
& & \\
\tau = \dfrac{c_{13}}{e_{12}}-\dfrac{e_{23}c_{14}}{e_{12}e_{24}}+\dfrac{c_{14}c_{21}c_{43}}{e_{12}e_{41}e_{24}} 
& \mbox{\hspace{2cm}} & \tilde{\tau} =  \dfrac{c_{14}}{e_{12}}-\dfrac{e_{24}c_{13}}{e_{12}e_{23}}+\dfrac{c_{13}c_{21}c_{34}}{e_{12}e_{31}e_{23}} \\
& & \\
\sigma = \dfrac{c_{14}}{e_{12}}\left(\dfrac{e_{23}}{e_{24}}-\dfrac{c_{13}}{c_{14}}\right) & \mbox{\hspace{2cm}} & \tilde{\sigma} = \dfrac{c_{13}}{e_{12}}\left(\dfrac{e_{24}}{e_{23}}-\dfrac{c_{14}}{c_{13}} \right)
\end{array}
$$
We preserve the use of  $\; \tilde{\mbox{}} \;$ to distinguish stability indices (and also the quantities $a_j$ and $b_j$) for the $\xi_3$-cycle from those for the 
$\xi_4$-cycle. The constants required in order to use Lemma \ref{P&A_B3} are given in Table 1. As in \cite{KS}, we make the following

\paragraph{Assumption 1:} Let $\rho, \tilde{\rho} > 1$ and $0<e_{24}/e_{23}<1$.

\begin{table}
\renewcommand{\arraystretch}{1.3}
\begin{tabular}{|c|c|c|}
\hline
\mbox{} & $\xi_3$-cycle & $\xi_4$-cycle \\
\hline
\hline
values at $\xi_1$ & $\tilde{a}_3 = c_{13}/e_{12}$; $\tilde{b}_3 = c_{14}/e_{12}$ & $a_3 = c_{14}/e_{12}$; $b_3 = c_{13}/e_{12}$ \\
\hline
values at $\xi_2$ & $\tilde{a}_1 = c_{21}/e_{23}$; $\tilde{b}_1 = -e_{24}/e_{23}$ & $a_1 = c_{21}/e_{24}$; $b_1 = -e_{23}/e_{24}$   \\
\hline
values at $\xi_3$ & $\tilde{a}_2 = c_{32}/e_{31}$; $\tilde{b}_2 = c_{34}/e_{31}$ & \mbox{}  \\
\hline
values at $\xi_4$ & \mbox{}  & $a_2 = c_{42}/e_{41}$; $b_2 = c_{43}/e_{41}$   \\
\hline
\end{tabular}
\caption{Parameter values at the nodes required in the calculation of the stability index for each cycle in the network. Notation of the eigenvalues as in \cite{KS}.}
\renewcommand{\arraystretch}{1}
\end{table}

\subsection{Another view  of ``A competition between heteroclinic cycles'' \cite{KS}}\label{another-view}

This subsection provides a systematic treatment of the network for the parameter values considered in \cite{KS}, calculating $c$-indices in both cases treated in \cite{KS}: when both cycles have $c$-indices greater than $-\infty$ and when one cycle has all $c$-indices equal to $-\infty$. In what concerns the $n$-indices, this section provides them, also in the case not addressed in \cite{KS} (that of $c$-indices greater than $-\infty$).

With the assumptions made in \cite{KS}, namely that $c_{ij}, e_{ij} > 0$ for all $i$ and $j$,  we need only look at case (iii) of Lemma \ref{P&A_B3}. Note that 

\begin{equation}\label{rho}
\rho = a_1a_2a_3 \quad \text{and} \quad \tilde{\rho}=\tilde{a}_1\tilde{a}_2\tilde{a}_3,
\end{equation}
as well as
\begin{equation}\label{delta}
\delta = b_1a_2a_3+b_3a_2+b_2 \quad \text{and} \quad \tilde{\delta}=\tilde{b}_1\tilde{a}_2\tilde{a}_3+\tilde{b}_3\tilde{a}_2+\tilde{b}_2.
\end{equation}
The following result provides generic information about the stability indices.

\begin{lemma}\label{minus_infinity}
All stability indices for the $\xi_3$- (respectively, $\xi_4$-) cycle are equal to $-\infty$ if and only if $\tilde{\delta}<0$ (respectively, $\delta <0$).
\end{lemma}

\begin{proof}
Straightforward given Lemma \ref{P&A_B3}. In fact, since $\rho, \tilde{\rho} > 1$, and given (\ref{rho}) and (\ref{delta}), we have all stability indices equal to $-\infty$ for the cycle corresponding to $\delta$ or $\tilde{\delta}$ negative.
\end{proof}

The following two propositions provide describe the $c$-indices of the cycles in the network.

\begin{proposition}\label{KS_fig5}
Under Assumption 1, let $\delta,\tilde{\delta}>0$. The stability indices are as follows:
\begin{eqnarray*}
\tilde{\sigma}_{23} = \sigma_{24} & = & +\infty \\
0 < \tilde{\sigma}_{12} & < & +\infty \\
-\infty < \sigma_{12} & < & 0
\end{eqnarray*}
and either $\tilde{\sigma}_{31} = +\infty$ 
or $\sigma_{41} = +\infty$, but not both. The remaining $c$-index is finite and can take either sign.
\end{proposition}

The case $\tilde{\sigma}_{31} = +\infty$  corresponds to the first two lines in Figure 5 of \cite{KS}, depending on the sign of $\sigma_{41}$, while the case $\sigma_{41} = +\infty$ corresponds to the last two lines of the same table, also depending on the sign of $\tilde{\sigma}_{31}$.

\begin{figure}[!htb]
\centerline{
\includegraphics[width=0.3\textwidth]{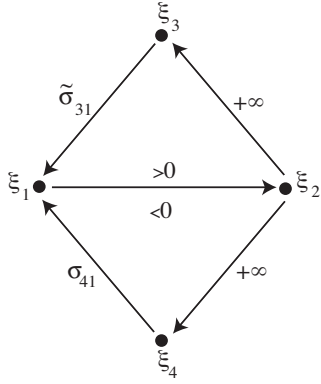}}
\caption{The stability indices for the network in Proposition \ref{KS_fig5}. Exactly one of $\sigma_{41}$ and $\tilde{\sigma}_{31}$ is equal to $+\infty$.}\label{Fig5_KS}
\end{figure}

\begin{proof}
We are looking at case (iii)(b), given the assumptions. In Figure $5$ of \cite{KS}, it is also $\delta, \tilde{\delta} > 0$. Then, from Lemma \ref{P&A_B3}, we obtain for the $\xi_3$-cycle
$$
\sigma_1 = \tilde{\sigma}_{12} = f^{\textnormal{index}}\left(\tilde{b}_1,1\right) = f^{\textnormal{index}}\left(-\frac{e_{24}}{e_{23}},1\right) = \frac{e_{23}}{e_{24}}-1 >0
$$
since $e_{24}<e_{23}$; $\sigma_2 = \tilde{\sigma}_{23} = +\infty$ and
$$
\sigma_3 = \tilde{\sigma}_{31} = f^{\textnormal{index}}\left(\tilde{b}_3+\tilde{b}_1\tilde{a}_3,1\right) = f^{\textnormal{index}}(-\tilde{\sigma},1) = \left\{ 
\begin{array}{ll}
+\infty & \mbox{if} \; \tilde{\sigma} \leq 0 \\
\frac{1}{\tilde{\sigma} }-1 > 0 & \mbox{if} \; \tilde{\sigma} \in (0,1) \\
1-\tilde{\sigma}  <0 & \mbox{if} \; \tilde{\sigma}>1
\end{array} \right. .
$$
For the $\xi_4$-cycle we obtain in a similar way
$$
\sigma_1 = \sigma_{12} = f^{\textnormal{index}}(b_1,1) = f^{\textnormal{index}}\left(-\frac{e_{23}}{e_{24}},1\right) = 1-\frac{e_{23}}{e_{24}} <0
$$
since $e_{24}<e_{23}$; $\sigma_2 = \sigma_{24} = +\infty$ and
$$
\sigma_3 = \sigma_{41} = f^{\textnormal{index}}(-\sigma,1) = \left\{ 
\begin{array}{ll}
+\infty & \mbox{if} \; \sigma \leq 0 \\
\frac{1}{\sigma}-1 > 0 & \mbox{if} \; \sigma \in (0,1) \\
1-\sigma <0 & \mbox{if} \; \sigma>1
\end{array} \right. .
$$

\end{proof}

If we relax the hypothesis that $\delta, \tilde{\delta} >0$, we have the following two possibilities.

\begin{proposition}\label{KS_lemma3}
Under Assumption 1, depending on the sign of $\delta$ the stability indices are as follows:
$$
\mbox{\bf{Case (i) ($\delta<0$):} } \;\; \sigma_{12}=\sigma_{24}=\sigma_{41}=-\infty, \; \; \; \tilde{\sigma}_{23}=\tilde{\sigma}_{34}=+\infty, \; \; \; \tilde{\sigma}_{12}>0
$$
and
$$
\mbox{\bf{Case (ii) ($\delta>0$):} } \;\; \sigma_{12}<0, \; \; \; \sigma_{24}=\sigma_{41}=+\infty, \; \; \; \tilde{\sigma}_{12}=\tilde{\sigma}_{23}=\tilde{\sigma}_{34}=-\infty.
$$
\end{proposition}

The proof is straightforward using Lemma \ref{P&A_B3} and omitted. 

\begin{figure}[!htb]
\centerline{
\includegraphics[width=0.7\textwidth]{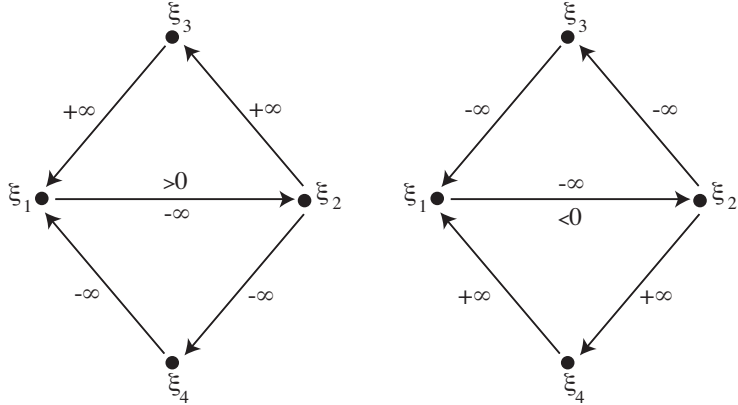}}
\caption{The stability indices for the network in Proposition \ref{KS_lemma3}. Case (i) is depicted on the left, whereas case (ii) appears on the right.}\label{Lemma3_KS}
\end{figure}

In Proposition \ref{KS_fig5}, all return maps around each cycle are contractions. In the case studied in Proposition \ref{KS_lemma3}, we have the competition between the cycles described by Kirk and Silber \cite{KS}. However, we note that, in either case, there is one cycle with all stability indices equal to $-\infty$. That is, one of the cycles attracts hardly anything in its neighbourhood. It follows from Lemma \ref{generic_plus_infinity} that then the other cycle has at least one stability index equal to $+\infty$, that is, the other cycle attracts almost all points near at least one of its connections.

\medskip

The competition between the cycles addressed in \cite{KS} becomes apparent when we calculate the stability indices with respect to the whole network. We include in Lemma \ref{KS_fig5-network} the case of the parameter values in Proposition \ref{KS_fig5}, not addressed in \cite{KS}, thus completing their study.

Recall that we use the superscript $n$ when the stability index is calculated with respect to the whole network. For the common trajectory we write simply $\sigma^n_{12}$.

\begin{lemma}\label{KS_fig5-network}
Under Assumption 1, let $\delta,\tilde{\delta}>0$. The stability indices, with respect to the network, are all positive. Furthermore, no connection with a finite $c$-index has an infinite $n$-index.
\end{lemma}

\begin{proof}
Because $\delta, \tilde{\delta}>0$ all return maps are contractions and we need only calculate along each connection the set of points taken outside $\dom(\tilde{h}_1) \cup \dom(h_1)$ by the local maps. All maps and domains from \cite{KS} can be found in the appendix.

By Lemma \ref{index_cycle-network} we know that $\tilde{\sigma}^n_{23} = \sigma^n_{24} = +\infty$ and $\sigma^n_{12}>0$. That $\sigma^n_{12}$ is finite follows from the observation that the union of the domains of definition of the return maps around each cycle starting at this connection excludes a cusp-shaped region. 

The proof proceeds by determining $\tilde{\sigma}^n_{31}$ and $\sigma_{41}$, case by case corresponding to each line of Figure 5 in \cite{KS}.

\paragraph{Line $1$:}  \hspace{.3cm}
Because of $\sigma >1$ and $\tilde{\sigma}<0$ we have $\sigma_{41}<0$ and $\tilde{\sigma}_{31} = +\infty$, which implies $\tilde{\sigma}^n_{31} = +\infty$. We calculate $\sigma^n_{41}$ by looking at the set of points 
\begin{align*}
\E_0 &= \{ (x,y) \in H_1^{\inn,4} | \; \; \phi_{412} (x,y) \notin \dom(\tilde{h}_1) \cup \dom(h_1) \}\\
&= \{ (x,y) \in H_1^{\inn,4} | \; \; (y x^{c_{13}/e_{12}}, x^{c_{14}/e_{12}}) \notin \dom(\tilde{h}_1) \cup \dom(h_1) \}\\
&= \{ (x,y) \in H_1^{\inn,4} | \; \; x^{-\tilde{\sigma}} < y < x^{-\tilde{\sigma}} \}.
\end{align*}
Since $-\tilde{\sigma} > 0$, the set $\E_0$ is a thin cusp and we have $\sigma^n_{41} > 0$ and finite.

\paragraph{Line $2$:}  \hspace{.3cm}
As in the previous case $\tilde{\sigma}^n_{31} = +\infty$. Moreover, $\sigma^n_{41} \geq \sigma_{41}>0$ and finite because of $0 < \sigma < 1$.

\paragraph{Line $3$:}  \hspace{.3cm}
In this case and the next we have $\sigma_{41} = +\infty$ due to $\sigma < 0$, so $\sigma^n_{41}=+\infty$. Because of $0 < \tilde{\sigma} < 1$ it follows that $\tilde{\sigma}^n_{31} \geq \tilde{\sigma}_{31} > 0$.

\paragraph{Line $4$:}  \hspace{.3cm}
Since $\tilde{\sigma} > 1$ we have $\tilde{\sigma}_{31}<0$ and thus determine $\tilde{\sigma}^n_{31}$ by calculating
\begin{align*}
\E_0 &= \{ (x,y) \in H_1^{\inn,3} | \; \; \phi_{312} (x,y) \notin \dom(\tilde{h}_1) \cup \dom(h_1) \}\\
&= \{ (x,y) \in H_1^{\inn,3} | \; \; (x^{c_{13}/e_{12}}, y x^{c_{14}/e_{12}}) \notin \dom(\tilde{h}_1) \cup \dom(h_1) \}\\
&= \{ (x,y) \in H_1^{\inn,3} | \; \; x^{\tilde{\sigma}} < y < x^{\tilde{\sigma}} \}.
\end{align*}
As in the first case this is a thin cusp and therefore we obtain $\tilde{\sigma}^n_{31} > 0$ and finite.
\end{proof}

When $\tilde{\sigma}_{31}>0$, corresponding to lines $1$, $2$ and $3$, the $\xi_3$-cycle is p.a.s.\ while the $\xi_4$-cycle is not. Hence, the $\xi_3$-cycle is trivially more visible in the network (it is more relatively stable), even though in the case corresponding to line $3$ the $\xi_4$-cycle has two stability indices equal to $+\infty$. Note that the behaviour along the common connecting trajectory is of the utmost importance in this case: although with respect to the network the stability index along this trajectory is positive, it is negative with respect to only the $\xi_4$-cycle and positive with respect to only the $\xi_3$-cycle, indicating that most points near this common connection are taken to $\xi_3$ rather than to $\xi_4$. The positive stability index with respect to the network merely informs that points not taken to $\xi_4$ still remain near the network (being taken to $\xi_3$ rather than away from the network).

Note also that, even when the $c$-index $\tilde{\sigma}_{31}$ is negative the $\xi_3$ cycle fails to be p.a.s.. However, the corresponding $n$-index, $\tilde{\sigma}^n_{31}$ is positive. Thus, being part of a network produces a stabilizing effect in the cycles in the sense that two non p.a.s.\ cycles give rise to a p.a.s.\ network.

\begin{lemma}\label{KS_lemma3-network}
Under Assumption 1, the stability indices, with respect to the network, are as follows:
$$
\mbox{\bf{Case (i) ($\delta<0$):} } \;\; \sigma_{12}^n,\sigma_{24}^n,\sigma_{41}^n>0, \; \; \; \tilde{\sigma}_{23}^n=\tilde{\sigma}_{31}^n=+\infty.
$$
$$
\mbox{\bf{Case (ii) ($\delta>0$):} } \;\; \sigma_{12}^n,\tilde{\sigma}_{23}^n,\tilde{\sigma}_{31}^n>0, \; \; \; \sigma_{24}^n=\sigma_{41}^n=+\infty.
$$
\end{lemma}

The proof is given in Lemma 3 of \cite{KS}.

For the parameter values of case (i) above the $\xi_3$-cycle is p.a.s.. It is therefore not surprising that the stability indices along the connections belonging to this cycle are all positive. In case (ii) none of the cycles is p.a.s.\ even though the network as a whole is p.a.s.. The attraction properties observed for the $\xi_4$-cycle are patent in the positive $n$-indices for the connections of this cycle. Again, the common connection is essential for understanding the visibility (or victory under competition) of the $\xi_4$-cycle. The fact that $\sigma_{12}^n>0$ ensures that most points remain near the network, while $\sigma_{12}>\tilde{\sigma}_{12}$ indicates that most of these points have trajectories which eventually come close to $\xi_4$. Again, case (ii) exhibits a stabilizing effect of the network.

\subsection{Positive transverse eigenvalues}\label{positive-transverse}

The existence of a common connecting trajectory ensures that, relative to one cycle, there is always a positive transverse eigenvalue at $\xi_2$. In this subsection, we consider the possibility of having another positive transverse eigenvalue, either at $\xi_3$ or at $\xi_4$. In order to do this, we admit that either $c_{34}$ or $c_{43}$ may be negative, thus creating a positive transverse eigenvalue at $\xi_3$ or $\xi_4$, respectively. We assume that the positive transverse eigenvalue is weaker than the expanding eigenvalue at the node. We further assume that $\tau, \tilde{\tau}, \delta, \tilde{\delta} >0$, to avoid the extreme case of  $c$-indices equal to $-\infty$. We thus have

\paragraph{Assumption 2:} Let $\tau, \tilde{\tau}, \delta, \tilde{\delta} >0$; $|c_{34}|<e_{31}$ at $\xi_3$ and $|c_{43}|<e_{41}$ at $\xi_4$.

\begin{proposition}\label{c34_negative}
Under Assumptions 1 and 2, if $c_{34}<0$, then $|\tilde{\sigma}_{12}|, |\sigma_{41}|<\infty$, $\tilde{\sigma}_{23} >0$, $\tilde{\sigma}_{31}=\sigma_{14}=+\infty$ and $\sigma_{12} <0$.
\end{proposition}

\begin{figure}[!htb]
\centerline{
\includegraphics[width=0.3\textwidth]{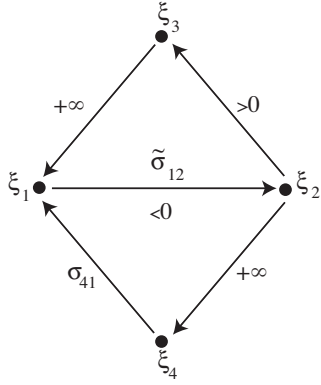}}
\caption{The $c$-indices for the network in Proposition \ref{c34_negative}. Stability indices $\tilde{\sigma}_{12}$ and $\sigma_{41}$ are finite but have no predetermined sign.}\label{negative_c34}
\end{figure}

\begin{proof}
The stability indices for the $\xi_4$-cycle are as in Proposition \ref{KS_fig5}.
For the $\xi_3$-cycle, given the imposed signs of the parameters, we are in case (iv)(b) of Lemma \ref{P&A_B3}. Since $\tilde{b}_1,\tilde{b}_2<0$, we have $f^{\textnormal{index}}\left(\tilde{b}_1+\tilde{b}_2\tilde{a}_1,1\right)<f^{\textnormal{index}}\left(\tilde{b}_1,1\right)$ and
$$
\sigma_1=\tilde{\sigma}_{12}=f^{\textnormal{index}}\left(\tilde{b}_1+\tilde{b}_2\tilde{a}_1,1\right)=\left\{
\begin{array}{ll}
-\frac{1}{\tilde{b}_1+\tilde{b}_2\tilde{a}_1}-1 > 0 & \mbox{if} \; \tilde{b}_1+\tilde{b}_2\tilde{a}_1 \in (-1,0) \\[5pt]
\tilde{b}_1+\tilde{b}_2\tilde{a}_1+1  <0 & \mbox{if} \; \tilde{b}_1+\tilde{b}_2\tilde{a}_1<-1
\end{array} \right. .
$$
The assumption that $|c_{34}|<e_{31}$ ensures that 
$$
\sigma_2=\tilde{\sigma}_{23}=-\frac{e_{31}}{c_{34}}-1 >0.
$$
Since we assume $\delta, \tilde{\delta} >0$, their expressions as functions of, respectively, $\tilde{\sigma}$ and $\sigma$ impose $\tilde{\sigma}<0$ and $\sigma >0$. Then $\tilde{\sigma}_{31}=+\infty$ and $|\sigma_{41}| < \infty$.
\end{proof}

We remark that in this case it is possible to choose signs for $\tilde{\sigma}_{12}$ and $\sigma_{41}$ so that the signs of the stability indices coincide for the connecting trajectories in both cycles: $\tilde{\sigma}_{12}< 0$ and $\sigma_{41}>0$. In this case, neither cycle is p.a.s.. Note that if $\tilde{\sigma}_{12}> 0$ then the $\xi_3$-cycle is p.a.s.\ and therefore always more visible.

\begin{proposition}\label{c43_negative}
Under Assumptions 1 and 2, if $c_{43}<0$, then $|\tilde{\sigma}_{31}| <\infty$, $\tilde{\sigma}_{12}, \sigma_{24} >0$, $\tilde{\sigma}_{23}=\sigma_{41} = +\infty$ and $\sigma_{12} <0$.
\end{proposition}

\begin{figure}[!htb]
\centerline{
\includegraphics[width=0.3\textwidth]{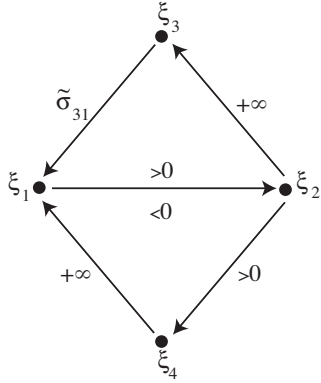}}
\caption{The $c$-indices for the network in Proposition \ref{c43_negative}. Stability index $\tilde{\sigma}_{31}$ is finite but has no predetermined sign.}\label{negative_c43}
\end{figure}

\begin{proof}
In this case, the stability indices for the $\xi_3$-cycle are as in Proposition \ref{KS_fig5}. 
For the $\xi_4$-cycle, given the imposed signs of the parameters, we are in case (iv)(b) of Lemma \ref{P&A_B3}. Since $b_1,b_2<0$, we have 
$$
\sigma_1=\sigma_{12}=f^{\textnormal{index}}(b_1+b_2a_1,1)=f^{\textnormal{index}}\left(-\frac{e_{23}}{e_{24}}+\frac{c_{43}c_{21}}{e_{41}e_{24}},1 \right)<0,
$$
where the inequality holds since $e_{23}>e_{24}$. The assumption that $|c_{43}|<e_{41}$ ensures that 
$$
\sigma_2=\tilde{\sigma}_{24}=-\frac{e_{41}}{c_{43}}-1 >0.
$$
Since we assume $\delta, \tilde{\delta} >0$, their expressions as functions of $\tilde{\sigma}$ and $\sigma$ impose $\tilde{\sigma}>0$ and $\sigma <0$. Then $\sigma_{41}=+\infty$ and $|\tilde{\sigma}_{31}| < \infty$.
\end{proof}

If $\tilde{\sigma}_{31} > 0$ then the $\xi_3$-cycle is again p.a.s., thus dominating the $\xi_4$-cycle within the network. Note however that $\tilde{\sigma}>1$ implies $\tilde{\sigma}_{31} < 0$, in which case each cycle has one stability index equal to $+\infty$, one positive and one negative. Neither cycle is p.a.s..

\medskip

In order to understand the relative stability of each cycle in the network and look for a stabilizing effect, 
we next calculate the stability indices with respect to the whole network in the cases where one of $c_{34}$ or $c_{43}$ is negative.
For $c_{34}<0$, an interesting case arises when we choose (see Proposition \ref{c34_negative})
$\tilde{\sigma}_{12} < 0$ and $\sigma_{41}>0$. This is the case when the cycles have, independently, the same collection of stability indices and neither is p.a.s.. In this case, as seen from the following theorem, the network is not p.a.s.. The negative value for the stability index along the common connection is preserved due to the existence of a positive transverse eigenvalue at $\xi_3$, preventing a stabilizing effect from taking place.

\begin{theorem}\label{c34_negative_network}
Under Assumptions 1 and 2, let $c_{34}<0$. The stability indices with respect to the network when $\tilde{\sigma}_{12} < 0$ and $\sigma_{41}>0$ are as follows:
$$
\tilde{\sigma}_{31}^n= \sigma_{24}^n=+\infty; \; \; \; \tilde{\sigma}_{23}^n, \sigma_{41}^n >0; \; \; \; \sigma_{12}^n <0.
$$
\end{theorem}

\begin{proof}
Changing the sign of $c_{34}$ does not affect the results on the return maps in Lemma 2 of \cite{KS}: If $\delta>0$, the return maps around the $\xi_4$-cycle are contractions. And if $\tilde{\delta}>0$, the return maps around the $\xi_3$-cycle are contractions. Their domains of definition, however, do not remain unchanged. Now that $c_{34}<0$, the local map $\phi_{231}$ is defined only on a cusp given by $y<x^{-\frac{c_{34}}{e_{31}}}$. All other points near the trajectory from $\xi_2$ to $\xi_3$ leave the neighbourhood of the network in the transverse direction. This obviously affects $\dom(\tilde{h}_2)$, restricting it by the same inequality, which shows that $\tilde{\sigma}_{23}^n$ is finite.

The domains of definition of the other local maps remain the same. But the change in $\dom(\phi_{231})$ influences the domains of all return maps around the $\xi_3$-cycle. For $\tilde{h}_1$ we now need to make sure that $\phi_{123}(.)$ lands in $\dom(\phi_{231})$. This changes the restriction on $\dom(\tilde{h}_1)$ from $y<x^\frac{e_{24}}{e_{23}}$ to $y < x^{\frac{e_{24}}{e_{23}}-\frac{c_{21}c_{34}}{e_{23}e_{31}}}$. The domain of $\tilde{h}_3$ has to be modified in the same way. Thus, the domains of the return maps around the $\xi_3$-cycle become
\begin{eqnarray*}
\tilde{h}_1: &  & y < x^{\frac{e_{24}}{e_{23}}-\frac{c_{21}c_{34}}{e_{23}e_{31}}} = x^{-\left(\tilde{b}_1+\tilde{b}_2\tilde{a}_1\right)} \\
\tilde{h}_2: & & y < x^{-\frac{c_{34}}{e_{31}}} \\
\tilde{h}_3: & & y < x^{-\tilde{\tau}}.
\end{eqnarray*}
The domains of the maps around the $\xi_4$-cycle are the same as before.

Points in the complement of $\dom(\tilde{h}_1) \cup \dom(h_1)$ inside an $\varepsilon$-ball in $H_1^{\out,2}$ satisfy
\begin{align}
x^{-\left(\tilde{b}_1+\tilde{b}_2\tilde{a}_1\right)} < y < x^{\frac{e_{24}}{e_{23}}}. \label{domains}
\end{align}
From the definition of the stability index, we obtain $\sigma_-(x) > 0$ so that $\sigma_{12}^n < 0$, when $-\left(\tilde{b}_1+\tilde{b}_2\tilde{a}_1\right)>1$. This is precisely the case when $\tilde{\sigma}_{12}<0$.

The preimage of the set of points that satisfy (\ref{domains}) under $\phi_{412}$ has positive measure in any $\varepsilon$-neighbourhood in $H_4^{\out,1}$, so $\sigma_{41}^n$ is finite.
\end{proof}

In this instance, a feature appears that has not been observed before: the common connection in the network has  negative $c$- and $n$-index, meaning that many trajectories stop following the network at this point; the network is not p.a.s..

In fact, we notice an interesting feature about the way in which the network may fail to be predominantly asymptotically stable. The sign of the stability index along the common trajectory is determined by $-\left(\tilde{b}_1+\tilde{b}_2\tilde{a}_1\right) \lessgtr 1$. It is positive when $-\left(\tilde{b}_1+\tilde{b}_2\tilde{a}_1\right) < 1$ which is the same as \mbox{$0>c_{34} > \frac{e_{31}(e_{24}-e_{23})}{c_{21}}$}. This means, as long as $c_{34}$ is negative, but not too small, the network is p.a.s.. In this instance, the stabilizing effect is apparent. However, once $c_{34}$ becomes smaller than $\frac{e_{31}(e_{24}-e_{23})}{c_{21}}$, we have $\sigma_{12}^n<0$ and neither the $\xi_3$-cycle nor the network is p.a.s.\ anymore. Thus, cycle and network lose predominant stability through an increasing transverse eigenvalue at $\xi_3$ --- but the actual loss of trajectories occurs along the connection from $\xi_1$ to $\xi_2$.

We now also give the $n$-indices for $c_{43}<0$. Again we focus on the most competitive case where the cycles have qualitatively equal indices. In contrast with what we found for the case $c_{34}<0$, we observe the existence of parameter values ensuring p.a.s.\ of the network while neither cycle is p.a.s., the network providing a stabilizing effect.

\begin{theorem}\label{c43_negative_network}
Under Assumptions 1 and 2, let $c_{43}<0$. The stability indices with respect to the network when $\tilde{\sigma}_{31} < 0$ are as follows:
$$
\tilde{\sigma}_{23}^n= \sigma_{41}^n=+\infty; \; \; \; \sigma_{12}^n, \sigma_{24}^n >0; \; \; \; |\tilde{\sigma}_{31}^n| <\infty.
$$
\end{theorem}

\begin{proof}
That $\tilde{\sigma}_{23}^n= \sigma_{41}^n=+\infty$ and $\sigma_{12}^n, \sigma_{24}^n >0$ follows from Lemma \ref{index_cycle-network}. The latter two indices are not equal to $+\infty$ because the domains of the respective return maps exclude a cusp-shaped region of points that move away from the network. Determining the domains is analogous to the previous theorem, yielding
\begin{eqnarray*}
h_1: &  & x < y^{\frac{e_{23}}{e_{24}}-\frac{c_{21}c_{43}}{e_{24}e_{41}}} = y^{-(b_1+b_2a_1)} \\
h_2: & & y < x^{-\frac{c_{43}}{e_{41}}} \\
h_4: & & y < x^{-\tau}.
\end{eqnarray*}
We need to investigate what happens along the connection from $\xi_3$ to $\xi_1$. A point $(x,y) \in H_3^{\out,1}$ belongs to the basin of attraction of the network if and only if $\phi_{312}(x,y) = (x^{\frac{c_{13}}{e_{12}}},yx^{\frac{c_{14}}{e_{12}}}) \in \dom(\tilde{h}_1) \cup \dom(h_1)$, which is equivalent to
\begin{align*}
yx^{\frac{c_{14}}{e_{12}}} < x^{\frac{c_{13}e_{24}}{e_{12}e_{23}}} \quad &\vee \quad x^{\frac{c_{13}}{e_{12}}} < \left( yx^{\frac{c_{14}}{e_{12}}} \right)^{\frac{e_{23}}{e_{24}}-\frac{c_{21}c_{43}}{e_{24}e_{41}}}\\
 \text{and thus} \quad y<x^{\tilde{\sigma}} \quad &\vee \quad y^{-\frac{e_{23}}{e_{24}}+\frac{c_{21}c_{43}}{e_{24}e_{41}}}<x^{-\tau}.
 \end{align*}
The first condition describes the thin side of a cusp, since $\tilde{\sigma}>1$ was the condition for $\tilde{\sigma}_{31}<0$. Whether the second condition describes the thin or the thick side of a cusp depends on $\alpha:=-\frac{e_{23}}{e_{24}}+\frac{c_{21}c_{43}}{e_{24}e_{41}}  \lessgtr -\tau$. For $\alpha<-\tau$, we have $\tilde{\sigma}^n_{31}<0$. Note that 
$$
\alpha<-\tau \Leftrightarrow \alpha <-\frac{c_{13}}{e_{12}+c_{14}} 
$$
and both this inequality and its reverse are compatible with $\tilde{\sigma}>1$ which is equivalent to 
$$
-\frac{c_{13}}{e_{12}+c_{14}}  < -\frac{e_{23}}{e_{24}}.
$$
\end{proof}

The stabilizing effect is apparent for $c_{43}<0$, but not too small, when the network may be p.a.s., but once $c_{43}$ is so small that $\alpha<-\tau$, the network loses its stability along the trajectory $[\xi_3 \to \xi_1]$. In either case, none of the individual cycles is p.a.s..

It is straightforward to see from the calculations in the previous proof that, when $\alpha>-\tau$, most points, that follow the network from a neighbourhood of $[\xi_3 \rightarrow \xi_1]$, end up in $\dom(h_1)$, that is, they switch from the $\xi_3$- to the $\xi_4$-cycle. Analogous calculations show that points $(x,y) \in H_4^{\out,1}$ will follow the network if $\phi_{412}(x,y) = (yx^{\frac{c_{13}}{e_{12}}},x^{\frac{c_{14}}{e_{12}}}) \in \dom(\tilde{h}_1) \cup \dom(h_1)$. We have $\phi_{412}(x,y) \in \dom(\tilde{h}_1)$ if
$$
y > x^{\sigma}
$$
and $\phi_{412}(x,y) \in \dom(h_1)$ if
$$
y<x^{-\tau}.
$$
Since $\sigma < 0$ to have the stability indices in Theorem \ref{c43_negative_network}, almost all points coming  from $[\xi_4 \rightarrow \xi_1]$ will remain close to the network by following the $\xi_4$-cycle.

\subsection{The stabilizing mechanism}\label{stabilizing}

We now focus on the existence of what we have been referring to as a stabilizing mechanism or effect, for positive transverse eigenvalues, in the extreme situation when one cycle has all $c$-indices equal to $-\infty$ and neither cycle is p.a.s.. We consider case (ii) in Proposition \ref{KS_lemma3} with the sole change given by $c_{34}<0$. We thus have $\tilde{\delta}<0<\delta$. Note that we no longer need to have $\sigma<0$ (which we did when $c_{34}>0$) and so $\sigma_{41}$ can have any sign or be equal to $+\infty$. The value of $\sigma_{24}$ remains unchanged and equal to $+\infty$. Thus, $\sigma^n_{24}=+\infty$ as well.

The main challenge, in terms of calculations, is that return maps no longer are contractions. The following lemmas provide the possible information about the $n$-index for each connection. The proof is analogous for the first three lemmas so we present only the first one in detail. The quantities required for these results are given next for ease of reference:

$$
\begin{array}{lcl}
\gamma_n = \tilde{\rho}^n\alpha-\tilde{\nu}\sum_{k=0}^{n-1}\tilde{\rho}^k & \hspace{2cm} &
\bar{\gamma}_n = \tilde{\rho}^n\frac{e_{24}}{e_{23}}-\tilde{\nu}\sum_{k=0}^{n-1}\tilde{\rho}^k   \\
& & \\
\zeta_n = -\tilde{\tau} \sum_{k=0}^{n}\tilde{\rho}^k  &\hspace{2cm} &
\bar{\zeta}_n = \tilde{\rho}^n\tilde{\sigma} -\tilde{\tau} \sum_{k=0}^{n-1}\tilde{\rho}^k   \\
& & \\
\eta_n = -\frac{c_{34}}{e_{31}}\tilde{\rho}^n-\tilde{\delta} \sum_{k=0}^{n}\tilde{\rho}^k &\hspace{2cm} &
\bar{\eta}_n = -\tilde{\delta} \sum_{k=0}^n\tilde{\rho}^k \\
\end{array}
$$

\begin{lemma}\label{sigma12}
Under Assumption 1, if there exists an $n \in \N$ such that $\bar{\gamma}_n<1<\gamma_n$ then $\sigma_{12}^n<0$. Otherwise, $0<\sigma_{12}^n<+\infty$.
\end{lemma}

\begin{proof}
Consider the set
\begin{align*}
\E_0&=(\dom(h_1) \cup \dom(\tilde{h}_1))^c \subset H_1^{\out,2},
\end{align*}
describing the points along the connection $[\xi_1 \rightarrow \xi_2]$ which are removed from a neighbourhood of the network.
The domains are given by the following inequalities, where $\alpha=\frac{e_{24}}{e_{23}}-\frac{c_{21}c_{34}}{e_{23}e_{31}}>0$
\begin{align*}
&\tilde{h}_1: \qquad y<x^\alpha\\
&h_1: \qquad x<y^{\frac{e_{23}}{e_{24}}}
\end{align*}
This gives
\begin{align*}
\E_0=\{(x,y) \mid x^\alpha < y < x^\frac{e_{24}}{e_{23}}   \}
\end{align*}
Note that $\frac{e_{24}}{e_{23}}<1$ and $\frac{e_{24}}{e_{23}}<\alpha$.
Define the preimages, describing points which leave a neighbourhood of the network after a finite number of iterates,
$$
\E_n:=\tilde{h}_1^{-n}(\E_0)
$$
which can be described by 
$$
\E_n=\{(x,y) \mid x^{\gamma_n} < y < x^{\bar{\gamma}_n} \}
$$
with $\gamma_n$ and $\bar{\gamma}_n$ as above. Both these sequences are monotonically increasing as can be seen by the fact that 
\begin{align*}
\gamma_{n+1}-\gamma_n&=(\tilde{\rho}^{n+1}-\tilde{\rho}^n)\alpha - \tilde{\nu}\tilde{\rho}^n\\
&=\tilde{\rho}^n \left( (\tilde{\rho}-1)\alpha-\frac{e_{24}}{e_{23}}(\tilde{\rho}-1)-\frac{c_{21}}{e_{23}}\tilde{\delta} \right)\\
&=\tilde{\rho}^n \left( (\tilde{\rho}-1)(\alpha-\frac{e_{24}}{e_{23}})-\frac{c_{21}}{e_{23}}\tilde{\delta} \right)\\
&=\tilde{\rho}^n \left( (\tilde{\rho}-1)(-\frac{c_{21}c_{34}}{e_{23}e_{31}})-\frac{c_{21}}{e_{23}}\tilde{\delta} \right)\\
&>0,
\end{align*}
and analogously for $\bar{\gamma}_{n+1} - \bar{\gamma}_n = -\frac{c_{21}}{e_{23}}\tilde{\delta}\tilde{\rho}^n >0$.

As long as either $\gamma_n, \bar{\gamma}_n<1$ or $\gamma_n, \bar{\gamma}_n>1$ the set $\E_n$ is a thin cusp-shaped region, and we get $\sigma_{12}^n>0$ by the standard argument from Lemma \ref{index_calculation}, which we use in these cases. However, if there is an $n \in \N$ such that $\bar{\gamma}_n<1<\gamma_n$, then $\E_n$ is a thick cusp, and we get $\sigma_{12}^n<0$ by similar arguments as above. 
\end{proof}

Note that a sufficient condition for $\sigma_{12}^n<0$ is that $\alpha > 1$ in which case we have $\bar{\gamma}_0 < 1 < \gamma_0$.

\begin{lemma}\label{sigma31}
If there exists an $n \in \N$ such that $\bar{\zeta}_n<1<\zeta_n$ then $\tilde{\sigma}_{31}^n<0$. Otherwise, $0<\tilde{\sigma}_{31}^n <+\infty$.
\end{lemma}

\begin{proof}
The points in $H_3^{\out,1}$ that do not stay near the network are those in
\begin{align*}
\F_0&=\{(x,y) \mid \phi_{312}(x,y) = (x^{\frac{c_{13}}{e_{12}}},yx^{\frac{c_{14}}{e_{12}}}) \notin \dom(\tilde{h}_1) \cup \dom(h_1)   \} \\
&=\{(x,y) \mid     x^{\alpha \frac{c_{13}}{e_{12}}} < yx^{\frac{c_{14}}{e_{12}}}  < x^{\frac{e_{24}}{e_{23}} \frac{c_{13}}{e_{12}}}      \}\\
&=\{(x,y) \mid     x^{\alpha \frac{c_{13}}{e_{12}} - \frac{c_{14}}{e_{12}}} < y  < x^{\frac{e_{24}}{e_{23}} \frac{c_{13}}{e_{12}} - \frac{c_{14}}{e_{12}}}      \}\\
&=\{(x,y) \mid     x^{-\tilde{\tau}} < y  < x^{\tilde{\sigma}}      \}
\end{align*}
Note that $c_{34}<0$ ensures $\tilde{\sigma}<-\tilde{\tau}$ so that $\F_0$ is non-empty.
The preimages of $\F_0$ under the return map $\tilde{h}_3$ are given by
$$
\zeta_n=-\tilde{\tau} \sum_{k=0}^{n}\tilde{\rho}^k \quad \text{and} \quad \bar{\zeta}_n= \tilde{\rho}^n\tilde{\sigma} -\tilde{\tau} \sum_{k=0}^{n-1}\tilde{\rho}^k
$$
where $\zeta_n$ and $\bar{\zeta}_n$ are as above. These sequences are again monotonically increasing and the result follows.
\end{proof}

\begin{lemma}\label{sigma23}
If there exists an $n \in \N$ such that $\bar{\eta}_n<1<\eta_n$ then $\tilde{\sigma}_{23}^n<0$. Otherwise, $0<\tilde{\sigma}_{23}^n <+\infty$.
\end{lemma}

\begin{proof}
The set of points in $H_2^{\out,3}$ that do not remain close to the network is
\begin{align*}
\G_0&=\{(x,y) \in \dom(\phi_{231}) \mid \phi_{312} \circ \phi_{231}(x,y)  \notin \dom(\tilde{h}_1) \cup \dom(h_1)   \} \\
&=\{(x,y) \in \dom(\phi_{231}) \mid (x^{\frac{c_{32}c_{13}}{e_{31}e_{12}}},yx^{\frac{c_{34}}{e_{31}} + \frac{c_{32}c_{14}}{e_{31}e_{12}}}) \notin \dom(\tilde{h}_1) \cup \dom(h_1)   \} \\
&=\{(x,y) \in \dom(\phi_{231}) \mid     x^{\alpha \frac{c_{32}c_{13}}{e_{31}e_{12}}} < yx^{\frac{c_{34}}{e_{31}} + \frac{c_{32}c_{14}}{e_{31}e_{12}}}  < x^{\frac{e_{24}}{e_{23}} \frac{c_{32}c_{13}}{e_{31}e_{12}}}      \}\\
&=\{(x,y) \in \dom(\phi_{231}) \mid     x^{-\tilde{\delta} - \frac{c_{34}}{e_{31}} \tilde{\rho}} < y  < x^{-\tilde{\delta}}      \}.
\end{align*}
Preimages under iteration of $\tilde{h}_2(x,y)=(x^{\tilde{\rho}},yx^{\tilde{\delta}})$ are of the form 
$$
\tilde{h}_2^{-n}(\G_0)=\{(x,y) \in \dom(\phi_{231}) \mid     x^{\eta_n} < y  < x^{\bar{\eta}_n}      \}
$$
with $\eta_n$ and $\bar{\eta}_n$ as above. Both sequences increase monotonically concluding the proof.
\end{proof}

\begin{lemma}\label{sigma41}
If $\sigma <0$ then $\tilde{\sigma}_{41}^n=+\infty$. If $\sigma \in (0,1)$ then $0<\tilde{\sigma}_{41}^n <+\infty$. If $\sigma >1$ then $\tilde{\sigma}_{41}^n<0$.
\end{lemma}

\begin{proof}
We want to look at
\begin{align*}
H_0&= \{(x,y) \mid \phi_{412}(x,y)=(yx^{\frac{c_{13}}{e_{12}}},x^{\frac{c_{14}}{e_{12}}}) \notin  \dom(\tilde{h}_1) \cup \dom(h_1)     \}\\
&= \{(x,y) \mid y^{\alpha}x^{\alpha \frac{c_{13}}{e_{12}}}   <  x^{\frac{c_{14}}{e_{12}}}  <   y^{\frac{e_{24}}{e_{23}}}x^{\frac{e_{24}}{e_{23}}\frac{c_{13}}{e_{12}}}  \}\\
&=\{ (x,y) \mid x^\sigma < y < x^\beta    \}
\end{align*}
where
$$
\beta= -\frac{1}{\alpha}\left(\alpha \frac{c_{13}}{e_{12}}-\frac{c_{14}}{e_{12}} \right)=\frac{1}{\alpha}\frac{c_{14}}{e_{12}} - \frac{c_{13}}{e_{12}}=\frac{\tilde{\tau}}{\alpha}.
$$
Since $\tilde{\tau}<0<\alpha$, we have $\beta<0$.\footnote{The fact that $\tilde{\tau}<0$ follows from the two ways of writing it: as $\tilde{\tau}=\tilde{\sigma}(\tilde{\rho}-1)+\frac{c_{13}c_{21}}{e_{12}e_{23}}\tilde{\delta}$ or as $\tilde{\tau} = -\tilde{\sigma}+\frac{c_{34}}{c_{32}}\tilde{\rho}$. If $\tilde{\tau}$ were positive then the first expression would imply $\tilde{\sigma}>0$, while the second would give $\tilde{\sigma}<0$.} Therefore, it does not place a restriction on the intersection of $H_0$ with small neighbourhoods. Since $h_4$ is a contraction $(\delta>0)$ we do not have to worry about preimages. Thus, following our earlier reasoning we get
\begin{itemize}
\item[(a)] If $\sigma<0$, then $H_0=\emptyset$ and $\sigma_{41}^n=+\infty$;
\item[(b)] If $\sigma \in (0,1)$, then $H_0$ is a thin cusp and $\sigma_{41}^n>0$ (but finite);
\item[(c)] If $\sigma > 1$, then $H_0$ is a thick cusp and $\sigma_{41}^n<0$.
\end{itemize}
\end{proof}

Notice that the sign of $\sigma_{41}^n$ depends on $\sigma$ in the same way as that of $\sigma_{41}$ in Proposition \ref{KS_fig5}.

We can now describe under which conditions the stabilizing effect of joining the cycles in a network is apparent to produce a p.a.s.\ network with two positive transverse eigenvalues.

\begin{proposition}
Under Assumption 1 and assuming $-c_{34}<e_{31}$, all $n$-indices are positive provided
$$
\sigma < \mbox{\textnormal{min}}\left\{ -\frac{e_{23}}{e_{24}}, -\frac{e_{23}}{e_{24}c_{32}}(e_{31}+c_{34})\right\}
$$
and
$$
1-\frac{c_{21}}{e_{23}} < \frac{e_{24}}{e_{23}} <1-\frac{c_{21}}{e_{23}}\frac{-c_{34}}{e_{31}}.
$$
\end{proposition}

\begin{proof}
A sufficient condition for $\tilde{\sigma}^n_{31}>0$ is $\tilde{\sigma}>1$. Then it is always $\zeta_n,\bar{\zeta}_n>1$. Note that $\tilde{\sigma}>1$ is equivalent to 
$$
\sigma<-\frac{e_{23}}{e_{24}}.
$$
Note that this condition is satisfied only if $\sigma <0$ in which case $\sigma_{41}=\sigma^n_{41} = +\infty$.

A sufficient condition for $\tilde{\sigma}^n_{23}>0$ is $\bar{\eta}_0=-\tilde{\delta}>1$. This is equivalent to
\begin{eqnarray*}
-\frac{c_{34}}{e_{31}}-\sigma \frac{c_{32}e_{24}}{e_{23}e_{31}} > 1 & \Leftrightarrow & \\
\Leftrightarrow \sigma \frac{c_{32}e_{24}}{e_{23}e_{31}} < -1 -\frac{c_{34}}{e_{31}} & \Leftrightarrow & \\
\Leftrightarrow \sigma < -\frac{e_{23}}{c_{32}e_{24}}(e_{31}+c_{34})
\end{eqnarray*}
The last inequality is compatible with the upper bound for $\sigma$.

A sufficient condition for $\tilde{\sigma}^n_{12}>0$ is that $\alpha<1$, providing the last inequality in the statement of the theorem, and $\bar{\gamma}_1>1$. We have, using the last upper bound for $\sigma$,
$$
\bar{\gamma}_1= \frac{e_{24}}{e_{23}}-\frac{c_{21}c_{34}}{e_{23}e_{31}}-\sigma\frac{c_{21}c_{34}c_{32}}{e_{23}^2e_{31}}
 > \frac{e_{24}}{e_{23}}+\frac{c_{21}}{e_{23}}
$$
and 
$$
\frac{e_{24}}{e_{23}}+\frac{c_{21}}{e_{23}} > 1 \Leftrightarrow 1-\frac{c_{21}}{e_{23}}< \frac{e_{24}}{e_{23}}.
$$
\end{proof}

\section{Stability of each cycle {\em versus} stability of the network}\label{stability_network} 

We have seen that it is possible to join two cycles to form a p.a.s.\ network in the following stability-related ways:
\begin{itemize}
	\item[(a)]  one p.a.s.\ and one non-p.a.s.\ cycle;
	\item[(b)]  two non-p.a.s.\ cycles.
\end{itemize}
We have encountered two instances of case (a): first, in the situation described by Proposition \ref{KS_fig5} and Lemma \ref{KS_fig5-network}, and then again, in case (i) of both Proposition \ref{KS_lemma3} and Lemma \ref{KS_lemma3-network}, where the non-p.a.s.\ cycle has all $c$-indices equal to $-\infty$. In case (ii) of both Proposition \ref{KS_lemma3} and Lemma \ref{KS_lemma3-network} we find an instance of (b). Another example illustrating case (b) appears in Proposition \ref{c34_negative} and Theorem \ref{c34_negative_network}. Case (b) illustrates a stabilizing effect of joining cycles into networks which appears common: even though none of the cycles is very stable (p.a.s.), the network as a whole attracts most points in its neighbourhood. This does not always happen, as can be seen from Theorem \ref{c43_negative_network}: with the choice made in the statement, neither cycle is p.a.s.; the network is not p.a.s.\ either for an open subset in parameter space.

As the next theorem demonstrates, a non-p.a.s.\ network can also be obtained by joining a p.a.s.\ cycle with a non-p.a.s.\ cycle. Thus, we have established the existence of all possible stability combinations for the two cycles and the network.

\begin{theorem}
A non-p.a.s.\ network may be constructed from a p.a.s.\ cycle and a non-p.a.s.\ cycle.
\end{theorem}

\begin{proof}
We find a network that is not p.a.s.\ even though one of its cycles is. Therefore, we look at the case $c_{34}<0$ as in Proposition \ref{c34_negative}. We choose
\begin{itemize}
\item $\tilde{b}_1+\tilde{b}_2\tilde{a}_1 \in (-1,0)$, such that $\tilde{\sigma}_{12}^n \geq \tilde{\sigma}_{12} >0$,
\item $\sigma > 1$, such that $\sigma_{41}<0$ by Proposition \ref{KS_fig5}.
\end{itemize}
The other stability indices are $\sigma_{24}=\tilde{\sigma}_{31}=+\infty$ and $\tilde{\sigma}_{23}>0$. So the $\xi_3$-cycle is p.a.s.\ and we need to find conditions such that $\sigma_{41}^n<0$.

From the calculations in the proof of Theorem \ref{c34_negative_network} we know that the complement of $\dom(h_1) \cup \dom(\tilde{h}_1)$ in $H_1^{\out,2}$ is given by
\begin{align}\label{complement}
x^{-\left(\tilde{b}_1+\tilde{b}_2\tilde{a}_1\right)}<y<x^{\frac{e_{24}}{e_{23}}}
\end{align}
So a sufficient condition for $\sigma^n_{41}<0$ is that the preimage of (\ref{complement}) under $\phi_{412}$ is the thick side of a cusp in $H_4^{\out,1}$. With $\phi_{412}(x,y)=(x^{\frac{c_{13}}{e_{12}}}y,x^{\frac{c_{14}}{e_{12}}})$ we find that $\phi_{412}(x,y)$ fulfilling (\ref{complement}) is equivalent to
\begin{align*}
x^{-\frac{c_{13}}{e_{12}}\left(\tilde{b}_1+\tilde{b}_2\tilde{a}_1\right)}y^{-\left(\tilde{b}_1+\tilde{b}_2\tilde{a}_1\right)}<x^{\frac{c_{14}}{e_{12}}}<x^{\frac{c_{13}e_{24}}{e_{12}e_{23}}}y^{\frac{e_{24}}{e_{23}}}\\
\Leftrightarrow x^{-\frac{c_{13}}{e_{12}}({-\frac{e_{24}}{e_{23}}+\frac{c_{34}c_{21}}{e_{31}e_{23}}})-\frac{c_{14}}{e_{12}}}<y^{-\frac{e_{24}}{e_{23}}+\frac{c_{34}c_{21}}{e_{31}e_{23}}} \quad \wedge \quad x^{\sigma}<y\\
\Leftrightarrow x^{-\tilde{\tau}} < y^{-\frac{e_{24}}{e_{23}}+\frac{c_{34}c_{21}}{e_{31}e_{23}}} \quad \wedge \quad x^{\sigma}<y
\end{align*}
Because of $\sigma>1$ the second condition describes the thick side of a cusp. The same is true for the first one if $-\tilde{\tau}(-\frac{e_{24}}{e_{23}}+\frac{c_{34}c_{21}}{e_{31}e_{23}})^{-1}>1$. This is the same as
\begin{align*}
-\tilde{\tau}<-\frac{e_{24}}{e_{23}}+\frac{c_{34}c_{21}}{e_{31}e_{23}}
\end{align*}
Both sides of the inequality are less than zero and this is compatible with $\tilde{\tau}>0$.
\end{proof}

We thus conclude with the observation that, even though the local stability index for heteroclinic connections is a useful tool, there is still a lot to be learnt about stability of heteroclinic networks.

\paragraph{Acknowledgements:}

The first author benefitted from financial support from the European Regional Development Fund through the programme COMPETE and from  the Portuguese Government through the Funda\c c\~ao para a Ci\^encia e a Tecnologia (FCT) under the project PEst-C/MAT/UI0144/2011. This research was partly carried out during a visit of the second author to the Centro de Matem\'atica da Universidade do Porto (CMUP) whose hospitality is gratefully acknowledged.

The second author thanks Reiner Lauterbach and Olga Podvigina for helpful discussions and email correspondence. Financial support of his aforementioned visit to Porto through the Nachwuchsfonds of the Mathematics Department at the University of Hamburg is also gratefully appreciated.

\appendix
\section{A simple $(B_2^+,B_2^+)$-network}
We construct a network with two $B_2^+$-cycles and provide the set-up for the study of the dynamics near the network by defining local and global maps.

Consider the finite Lie group $\Z_2^3$ generated by the following elements of order $2$:
\begin{eqnarray*}
\kappa_2 . (x_1,x_2,x_3,x_4) & = & (x_1,-x_2,x_3,x_4) \\
\kappa_3 . (x_1,x_2,x_3,x_4) & = & (x_1,x_2,-x_3,x_4) \\
\kappa_4 . (x_1,x_2,x_3,x_4) & = & (x_1,x_2,x_3,-x_4).
\end{eqnarray*}
It is easily seen that, for $i \neq j$, Fix$(\langle \kappa_i, \kappa_j \rangle)$ is a two-dimensional space of the form
$$
P_{1k} = \{ (x_1,x_2,x_3,x_4) \in \R^4: \; x_i=x_j=0, \; k \neq i,j \}.
$$
We further have Fix$(\langle \kappa_2,\kappa_3,\kappa_4 \rangle)=L_1=\{ (x_1,0,0,0): \; x_1 \in \R \}$. Let $f$ be a vector field equivariant under this group action. Then, when restricted to one of the invariant planes, the vector field has the form (we write the equations in $P_{12}$ for concreteness):
$$
\left\{ \begin{array}{l}
\dot{x_1} = a_1x_1+b_1(x_1^2+x_2^2)+c_1x_1^3 \\
\dot{x_2} = a_2x_2 +b_2(x_1^2+x_2^2)x_2 + d_1x_1x_2
\end{array} \right. .
$$

The origin is always an equilibrium. Assume $a_2b_2>0$ so that there are no equilibria on the $x_2$-axis. Assume further that $b_1^2-4a_1c_1 >0$ so that there are two equilibria, other than the origin, on the $x_1$-axis. Set $c_1<0$ and label these $\xi_a$ and $\xi_b$, where the first coordinate of $\xi_a$ is negative and the first coordinate of $\xi_b$ is positive. Choose $a_1,a_2>0$ so that the origin is a source and the remaining coefficients so that $\xi_a$ is a saddle and $\xi_b$ is a sink in $P_{12}$. In $P_{13}$ and $P_{14}$ coefficients can be chosen so that $\xi_a$ is a sink and $\xi_b$ a saddle.

We thus obtain a heteroclinic cycle made of three connections as follows:
$$
[\xi_{a} \rightarrow \xi_{b}]  \mbox{ in  } P_{12} ; \; \; \;
[\xi_{b} \rightarrow \xi_{a}]  \mbox{ in  } P_{13}  ; \; \; \; 
[\xi_{b} \rightarrow \xi_{a}]  \mbox{ in  } P_{14} .
$$
There are two cycles: $C_3 = [\xi_{a} \rightarrow \xi_{b}\rightarrow \xi_{a}] \subset P_{12} \cup P_{13}$ and 
$C_4 = [\xi_{a} \rightarrow \xi_{b}\rightarrow \xi_{a}] \subset P_{12} \cup P_{14}$. These correspond to the $\xi_3$- and $\xi_4$-cycles of \cite{KS}.

\begin{figure}[!htb]
\centerline{
\includegraphics[width=1.0\textwidth]{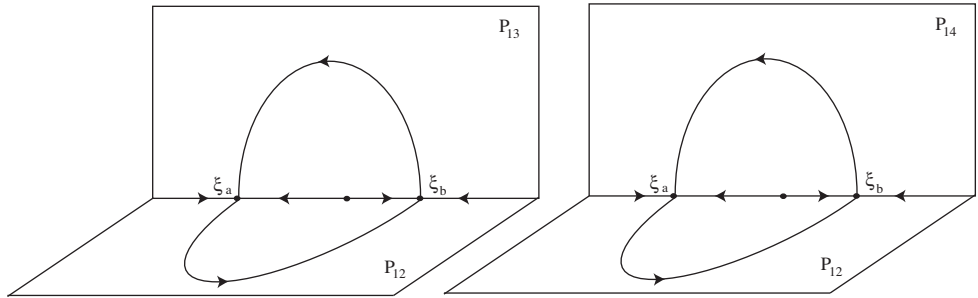}}
\caption{The $B_2^+$-cycles in the network.}\label{cyclesB2}
\end{figure}

\paragraph{Dynamics near the network of $B_2^+$-cycles:} In a way analogous to that used by Kirk and Silber \cite{KS}, we use the linearization at each equilibrium on the network to define local maps. Global maps are defined as small perturbations of the identity, conditioning the domain of definition of the return maps around each cycle in the network.

Near $\xi_a$ the local maps are defined for points in an incoming section to the flow approaching $\xi_a$, $H_{ai}^{\inn}$ for $i=3,4$, with image in an outgoing section to the flow leaving $\xi_a$, $H_{a2}^{\out}$. Linearize the flow near $\xi_a$ to obtain
$$
\left\{ \begin{array}{l}
\dot{x_1} = -r_ax_1 \\
\dot{x_2} = e_{a2}x_2 \\
\dot{x_3} = -c_{a3}x_3 \\
\dot{x_4} = -c_{a4}x_4 ,
\end{array} \right. 
$$
where all the constants are positive.

Near $\xi_b$, the local maps are analogously defined but now we have two outgoing sections, $H_{b3}^{\out}$ and $H_{b4}^{\out}$, and one incoming section, $H_{b2}^{\inn}$. Linearization of the flow near $\xi_b$ provides
$$
\left\{ \begin{array}{l}
\dot{x_1} = -r_bx_1 \\
\dot{x_2} = -c_{b2}x_2 \\
\dot{x_3} = e_{b3}x_3 \\
\dot{x_4} = e_{b4}x_4 ,
\end{array} \right. 
$$
where again all the constants are positive.
Assume from now on, and without loss of generality, that $e_{b3}>e_{b4}$.

The coordinates for the sections to the flow are as follows:
\begin{eqnarray*}
H_{a2}^{\out} = H_{b2}^{\inn} & = & \{ (x_1,1,x_3,x_4) \} \\
H_{a3}^{\inn} = H_{b3}^{\out} & = & \{ (x_1,x_2,1,x_4) \} \\
H_{a4}^{\inn} = H_{b4}^{\out} & = & \{ (x_1,x_2,x_3,1) \} .
\end{eqnarray*}

The standard construction of the local maps using the linearized flow gives the following:
\begin{eqnarray*}
\varphi_{a3}: H_{a3}^{\inn} \to H_{a2}^{\out}, \quad \varphi_{a3}(x_1,x_2,1,x_4) & = & \Big(x_1 x_2^{\frac{r_a}{e_{a2}}}, 1, x_2^{\frac{c_{a3}}{e_{a2}}}, x_4 x_2^{\frac{c_{a4}}{e_{a2}}}  \Big) \\
\varphi_{b3}: H_{b2}^{\inn} \to H_{b3}^{\out}, \quad \varphi_{b3}(x_1,1,x_3,x_4) & = & \Big(x_1 x_3^{\frac{r_b}{e_{b3}}}, x_3^{\frac{c_{b2}}{e_{b3}}}, 1, x_4 x_3^{-\frac{e_{b4}}{e_{b3}}} \Big) \\
\varphi_{a4}: H_{a4}^{\inn} \to H_{a2}^{\out}, \quad \varphi_{a4}(x_1,x_2,x_3,1) & = & \Big(x_1 x_2^{\frac{r_a}{e_{a2}}}, 1, x_3 x_2^{\frac{c_{a3}}{e_{a2}}}, x_2^{\frac{c_{a4}}{e_{a2}}} \Big) \\
\varphi_{b4}: H_{b2}^{\inn} \to H_{b3}^{\out}, \quad \varphi_{b4}(x_1,1,x_3,x_4) & = & \Big(x_1 x_4^{\frac{r_b}{e_{b4}}}, x_4^{\frac{c_{b2}}{e_{b4}}}, x_3 x_4^{-\frac{e_{b3}}{e_{b4}}}, 1 \Big)
\end{eqnarray*}
The domain of definition of the maps $\varphi_{b3}$ and $\varphi_{b4}$ is, respectively, constrained by the inequalities
$$
(1-\epsilon)x_3^{\frac{e_{b4}}{e_{b3}}} > x_4 \geq 0 \; \; \mbox{   and     }\; \; \;  (1-\epsilon)x_4^{\frac{e_{b3}}{e_{b4}}} > x_3 \geq 0,
$$
in the respective (local) coordinates. By composing these local maps with global maps analogous to those used by \cite{KS}, without resorting to polar coordinates however, we obtain four return maps, one for each connection belonging to each cycle. These are, for $C_3$,
\begin{eqnarray*}
g_{3a} & : & H_{a3}^{\inn} \rightarrow H_{a3}^{\inn} \\
g_{3b} & : & H_{b2}^{\inn} \rightarrow H_{b2}^{\inn}
\end{eqnarray*}
and, for $C_4$,
\begin{eqnarray*}
g_{4a} & : & H_{a4}^{\inn} \rightarrow H_{a4}^{\inn} \\
g_{4b} & : & H_{b2}^{\inn} \rightarrow H_{b2}^{\inn}.
\end{eqnarray*}

The return maps are given by:
$$
g_{3a}(x_1,x_2,1,x_4) = \Big(A_1 x_1 x_2^{\frac{r_a}{e_{a2}}+\frac{c_{a3}r_b}{e_{a2}e_{b3}}}, B_1 x_2^{\tilde{\rho}},1, C_1 x_4 x_2^{\tilde{\delta}}\Big), 
$$
with $0 \leq x_4 < k_1^{3a} (1-\epsilon)x_2^{\frac{c_{a3}}{e_{a2}}(\frac{e_{b4}}{e_{b3}}-\frac{c_{a4}}{c_{a3}})}$;
$$
g_{3b}(x_1,1,x_3,x_4) = \Big(A_2 x_1 x_3^{\frac{r_b}{e_{b3}}+\frac{c_{b2}r_a}{e_{a2}e_{b3}}}, 1, B_2 x_3^{\tilde{\rho}}, C_2 x_4 x_3^{\frac{e_{b4}}{e_{b3}}(\rho-1)}\Big), 
$$
with $0 \leq x_4 < k_2^{3a}(1-\epsilon)x_3^{\frac{e_{b4}}{e_{b3}}}$;
$$
g_{4a}(x_1,x_2,x_3,1) = \Big(D_1 x_1 x_2^{\frac{r_a}{e_{a2}}+\frac{c_{a4}r_b}{e_{a2}e_{b4}}}, E_1 x_2^{\rho}, x_3 x_2^{\delta}, 1 \Big),
$$
with $0 \leq x_3 < k_2^{3a}(1-\epsilon)x_2^{\frac{c_{a4}}{e_{a2}}(\frac{e_{b3}}{e_{b4}}-\frac{c_{a3}}{c_{a4}})}$;
$$
g_{4b}(x_1,1,x_3,x_4) = \Big(D_2 x_1 x_4^{\frac{r_b}{e_{b4}}+\frac{c_{b2}r_a}{e_{b4}e_{a2}}}, 1, E_2 x_3 x_4^{\frac{e_{b3}}{e_{b4}}(\tilde{\rho}-1)}, F_2 x_4^{\rho}\Big),
$$
with $0 \leq x_3 < k_2^{3a}(1-\epsilon)x_4^{\frac{e_{b3}}{e_{b4}}}$,
where
\begin{align*}
\rho&:=\frac{c_{a4}c_{b2}}{e_{a2}e_{b4}}, \quad \delta:=\frac{c_{a3}}{e_{a2}}-\frac{e_{b3}c_{a4}}{e_{a2}e_{b4}}\\
\tilde{\rho}&:=\frac{c_{a3}c_{b2}}{e_{a2}e_{b3}}, \quad \tilde{\delta}:=\frac{c_{a4}}{e_{a2}}-\frac{e_{b4}c_{a3}}{e_{a2}e_{b3}}.
\end{align*}

Notice that $\delta \tilde{\delta} <0$.

\paragraph{Stability indices:}  In the terminology of Lemma \ref{P&A_B2} we have $\tilde{\rho}=\tilde{a}_1\tilde{a}_2$ for $C_3$ and $\rho=a_1a_2$ for $C_4$, where again we use $\; \tilde{\mbox{}} \;$ to distinguish between the ratios of eigenvalues for the respective cycles. Since we do not want all indices along one of the cycles to be equal to $-\infty$, from now on we assume $\rho,\tilde{\rho}>1$. Note that precisely one of $\tilde{\delta}$ and $\delta$ is positive. In Theorem \ref{stab_index_B2} we give the stability indices for cases \mbox{(i) $\delta<0$ ($\Rightarrow \tilde{\delta}>0$)} and \mbox{(ii) $\delta>0$ ($\Rightarrow \tilde{\delta}<0$)}.

Subscripts indicate the direction of the connection and the cycle: $\sigma_{ij,3}$ for the connection in $C_3$ and the stability index relative only to this cycle, whereas we write $\sigma_{ij,3}^{n}$ for the stability index along the same connection but now calculated for the network.
Note that, when calculating the stability index of the common connection with respect to the network, we have $\sigma_{ab,3}^n=\sigma_{ab,4}^n$. In this case, we use $\sigma_{ab}^n$.

\begin{theorem}\label{stab_index_B2}
Generically, the stability indices for connecting trajectories in the network are 
$$
\mbox{\bf{Case (i) ($\delta<0$):} } \;\; \sigma_{ab,4}=\sigma_{ba,4}=-\infty, \; \; \; \sigma_{ab,3}>0, \; \; \; \sigma_{ba,3}=+\infty;
$$
$$
\sigma_{ab}^n,\sigma_{ba,4}^n>0, \; \; \; \sigma_{ba,3}^n=+\infty.
$$
$$
\mbox{\bf{Case (ii) ($\delta>0$):} } \;\; \sigma_{ab,4}<0, \; \;\; \sigma_{ba,4}=+\infty, \; \; \; \sigma_{ab,3}=\sigma_{ba,3}=-\infty;
$$
$$
\sigma_{ab}^n,\sigma_{ba,3}^n>0, \; \; \; \sigma_{ba,4}^n=+\infty.
$$
\end{theorem}

\begin{proof}
The stability indices relative to the cycles follow from Lemma \ref{P&A_B2} (iii), where $\tilde{a}_1\tilde{a}_2=\tilde{\rho}$ and  $\tilde{b}_1\tilde{a}_2+\tilde{b}_2=\tilde{\delta}$ for $C_3$, while $a_1a_2=\rho$ and $b_1a_2+b_2=\delta$ for $C_4$.

For the indices with respect to the network, we show how to calculate $\sigma_{ba,3}^n$ in case (ii), the others can be determined in a similar manner. In case (ii) we have $\delta>0$, which together with $\rho,\tilde{\rho}>1$ implies that the return maps around $C_4$ are contractions.

This allows us to determine all points in $H_{a3}^{\inn}$, that are not attracted to the network, in two steps: First we calculate the preimage $\E_0 \subset H_{a3}^{\inn}$ under $\varphi_{a3}$ of the complement of $\dom(g_{3b})\cup \dom(g_{4b})$, the union of the domains of the return maps $g_{3b}$ and $g_{4b}$ in $H_{b2}^{\inn}$. Then we take the union of preimages of $\E_0$ under the return map $g_{3a}$, $\E_n:=g_{3a}^{-n}(\E_0)$. In the same way as the authors of \cite{KS} do, we restrict the calculations (and notation) to the two relevant components. Also, we adjust the constants $k, \hat{k}$ in every step.
\begin{align*}
\E_0&=\Big\{(x_2,x_4) \in H_{a3}^{\inn} \mid \varphi_{a3}(x_2,x_4) \notin D(g_{3b})\cup D(g_{4b}) \Big\}\\
&=\Big\{(x_2,x_4) \in H_{a3}^{\inn} \mid \Big(x_2^{\frac{c_{a3}}{e_{a2}}},x_4x_2^{\frac{c_{a4}}{e_{a2}}}\Big) \notin D(g_{3b})\cup D(g_{4b}) \Big\}\\
&=\Big\{(x_2,x_4) \in H_{a3}^{\inn} \mid  kx_2^{\frac{e_{b4}}{e_{b3}}\frac{c_{a4}}{e_{a2}}} \leq x_4x_2^{\frac{c_{a4}}{e_{a2}}} \leq \hat{k}x_2^{\frac{e_{b4}}{e_{b3}}\frac{c_{a4}}{e_{a2}}} \Big\}\\
&=\Big\{(x_2,x_4) \in H_{a3}^{\inn} \mid kx_2^{-\tilde{\delta}} \leq x_4 \leq \hat{k}x_2^{-\tilde{\delta}} \Big\}\\
\Rightarrow \E_1&=\Big\{ (x_2,x_4) \in H_{a3}^{\inn} \mid g_{3a}(x_2,x_4) \in \E_0 \Big\}\\
&= \Big\{ (x_2,x_4) \in H_{a3}^{\inn} \mid (B_1x_2^{\tilde{\rho}},C_1x_4x_2^{\tilde{\delta}}) \in \E_0 \Big\}\\
&= \Big\{ (x_2,x_4) \in H_{a3}^{\inn} \mid kx_2^{-\tilde{\delta}\tilde{\rho}-\tilde{\delta}} \leq x_4 \leq \hat{k}x_2^{-\tilde{\delta}\tilde{\rho}-\tilde{\delta}} \Big\}\\
\intertext{Iteration leads to}
\E_n&=\Big\{ (x_2,x_4) \in H_{a3}^{\inn} \mid kx_2^{\alpha_n} \leq x_4 \leq \hat{k}x_2^{\alpha_n}\Big\},\\
\text{where} \quad \alpha_n&=-\tilde{\delta} \sum_{j=0}^{n} \tilde{\rho}^j.
\end{align*}
The sequence of exponents $(\alpha_n)_{n \in \N}$ is monotonically increasing and unbounded since $\alpha_{n+1}-\alpha_n=-\tilde{\delta}\tilde{\rho}^{n+1}>0$. Therefore, in the generic case $\alpha_n \neq 1$ for all $n \in \N$, by Lemma \ref{index_calculation} we obtain $\sigma_{ba,3}^n>0$.

For the calculation of $\sigma_{ab}^n$ the sequence of exponents turns out to be $\beta_n=\frac{e_{b4}}{e_{b3}}\tilde{\rho}^n-(\rho-1)\sum_{j=0}^{n-1}\tilde{\rho}^j$. This gives $\beta_{n+1}-\beta_n=\frac{e_{b4}}{e_{b3}}\tilde{\rho}^n(\tilde{\rho}-\rho)$ and since $\delta>0 \Leftrightarrow \tilde{\rho}>\rho$, this sequence is also monotonically increasing, giving $\sigma_{ab}^n>0$.
\end{proof}

\section{Maps between cross-sections}
For the $(B_3^-,B_3^-)$-network the standard construction using the linearized flow gives the following local maps. They have been determined by \cite{KS}, but not all of them are listed explicitly in their work.
\begin{eqnarray*}
\phi_{123}: H_{2}^{\inn,1} \to H_{2}^{\out,3}, \quad \phi_{123}(1,x_2,x_3,x_4) & = & \Big(x_3^{\frac{c_{21}}{e_{23}}}, x_2x_3^{\frac{r_2}{e_{23}}},1, x_4 x_3^{-\frac{e_{24}}{e_{23}}}  \Big) \\
\phi_{231}: H_{3}^{\inn,2} \to H_{3}^{\out,1}, \quad \phi_{231}(x_1,1,x_3,x_4) & = & \Big(1, x_1^{\frac{c_{32}}{e_{31}}}, x_3x_1^{\frac{r_3}{e_{31}}}, x_4 x_1^{\frac{c_{34}}{e_{31}}} \Big) \\
\phi_{312}: H_{1}^{\inn,3} \to H_{1}^{\out,2}, \quad \phi_{312}(x_1,x_2,1,x_4) & = & \Big(x_1 x_2^{\frac{r_1}{e_{12}}}, 1, x_2^{\frac{c_{13}}{e_{12}}}, x_4x_2^{\frac{c_{14}}{e_{12}}} \Big) \\
\phi_{124}: H_{2}^{\inn,1} \to H_{2}^{\out,4}, \quad \phi_{124}(1,x_2,x_3,x_4) & = & \Big(x_4^{\frac{c_{21}}{e_{24}}}, x_2x_4^{\frac{r_2}{e_{24}}}, x_3 x_4^{-\frac{e_{23}}{e_{24}}},1  \Big) \\
\phi_{241}: H_{4}^{\inn,2} \to H_{4}^{\out,1}, \quad \phi_{241}(x_1,1,x_3,x_4) & = & \Big(1, x_1^{\frac{c_{42}}{e_{41}}}, x_3x_1^{\frac{c_{43}}{e_{41}}}, x_4 x_1^{\frac{r_4}{e_{41}}} \Big) \\
\phi_{412}: H_{1}^{\inn,4} \to H_{1}^{\out,2}, \quad \phi_{412}(x_1,x_2,x_3,1) & = & \Big(x_1 x_2^{\frac{r_1}{e_{12}}}, 1, x_3x_2^{\frac{c_{13}}{e_{12}}}, x_2^{\frac{c_{14}}{e_{12}}} \Big) 
\end{eqnarray*}
Only $\phi_{123}$ and $\phi_{124}$ are not defined on a whole neighbourhood of the trajectory. Their domains of definition within the transverse section $H_{2}^{\inn,1}$ are bounded by the inequalities $x_4<x_3^{\frac{e_{24}}{e_{23}}}$ and $x_3<x_4^{\frac{e_{23}}{e_{24}}}$, respectively. Combining local and global maps in the appropriate order gives the return maps as found in \cite{KS}. Reduced to the two components relevant for stability, these are:
\begin{eqnarray*}
\tilde{h}_1: H_{1}^{\out,2} \to H_{1}^{\out,2}, \quad \tilde{h}_1(x,y) & = & (x^ {\tilde{\rho}},yx^{\tilde{\nu}})  \quad \text{for} \quad y<x^{\frac{e_{24}}{e_{23}}}\\
\tilde{h}_2: H_{2}^{\out,3} \to H_{2}^{\out,3}, \quad \tilde{h}_2(x,y) & = & (x^{\tilde{\rho}},yx^{\tilde{\delta}})  \quad \text{for} \quad y<x^{-\tilde{\delta}}\\
\tilde{h}_3: H_{3}^{\out,1} \to H_{3}^{\out,1}, \quad \tilde{h}_3(x,y) & = & (x^{\tilde{\rho}},yx^{\tilde{\tau}})  \quad \text{for} \quad y<x^{\tilde{\sigma}}\\
h_1: H_{1}^{\out,2} \to H_{1}^{\out,2}, \quad h_1(x,y) & = & (xy^\nu,y^\rho) \quad \text{for} \quad x<y^{\frac{e_{23}}{e_{24}}}\\
h_2: H_{2}^{\out,4} \to H_{2}^{\out,4}, \quad h_2(x,y) & = & (x^\rho,yx^\delta) \quad \text{for} \quad y<x^{-\delta}\\
h_4: H_{4}^{\out,1} \to H_{4}^{\out,1}, \quad h_4(x,y) & = & (x^\rho,yx^\tau) \quad \text{for} \quad y<x^{\sigma}
\end{eqnarray*}
Note that we ignore technicalities such as constant coefficients induced by non-identity global maps. These do not influence the calculation of stability indices.

\end{document}